\documentclass[11pt]{article}
\usepackage[a4paper]{anysize}\marginsize{3.5cm}{3.5cm}{1.3cm}{2cm}
\pdfpagewidth=\paperwidth \pdfpageheight=\paperheight
\usepackage{amsfonts,amssymb,amsthm,amsmath,amsxtra,amscd,verbatim,eucal}
\usepackage{pgf}
\usepackage[textures,bookmarks=false]{hyperref}
\usepackage[all]{xy}
\usepackage[active]{srcltx}

\theoremstyle{plain}
\newtheorem{thm}{Theorem}[section]
\newtheorem{theorem}[thm]{Theorem}

\newtheorem{lemma}[thm]{Lemma}
\newtheorem{proposition}[thm]{Proposition}
\newtheorem{corollary}[thm]{Corollary}

\theoremstyle{definition}
\newtheorem{definition}[thm]{Definition}
\newtheorem{remark}[thm]{Remark}
\newtheorem{example}[thm]{Example}

\renewcommand\phi{\varphi}
\renewcommand\ge{\geqslant}
\renewcommand\geq{\geqslant}
\renewcommand\le{\leqslant}
\renewcommand\leq{\leqslant}

\let\tilde=\widetilde

\newcommand\be{\begin{eqnarray*}}
\newcommand\ee{\end{eqnarray*}}

\newcommand\Q{\mathbb Q}
\newcommand\R{\mathbb R}
\newcommand\C{\mathbb C}
\newcommand\Z{\mathbb Z}
\newcommand\N{\mathbb N}
\renewcommand\P{\mathbb P}

\newcommand\calo{{\mathcal O}}
\newcommand\shs{{\mathcal S}}

\newcommand\calf{{\mathcal F}}
\newcommand\newop[2]{\def#1{\mathop{\rm #2}\nolimits}}
\newop\log{log}
\newop\clconv{clconv}
\newop\cl{cl}
\newop\ord{ord}
\newop\Gal{Gal}
\newop\SL{SL}
\newop\Bl{Bl}
\newop\mult{mult}
\newop\mass{mass}
\newop\div{div}
\newop\codim{codim}
\newop\inte{int}
\newcommand\eqnref[1]{(\ref{#1})}

\newcommand\fbul{\calf_{\bullet}}
\newcommand\fmbul{\calf_{m\bullet}}
\newcommand\ebul{E_{\bullet}}
\newcommand\vbul{V_{\bullet}}
\newcommand\vmbul{V_{m\bullet}}
\newcommand\ybul{Y_{\bullet}}
\newcommand\calfbul{\calf_{\bullet}}
\DeclareMathOperator{\Val}{Val}
\newcommand\valybul{\Val_{\ybul}}
\newcommand\dybul{\Delta_{\ybul}}
\newcommand\emin{e_{\min}}
\newcommand\emax{e_{\max}}
\newcommand\nuy{\nu_{\ybul}}
\newcommand\phifbul{\varphi_{\fbul}}
\newcommand\psifbul{\psi_{\fbul}}
\newcommand\phifmbul{\varphi_{\fmbul}}
\newcommand\wtilde[1]{\widetilde{#1}}
\newcommand\wphifbul{\wtilde{\phifbul}}
\newcommand\restr[1]{\big|_{#1}}
\newcommand{\st}[1]{\ensuremath{ \left\{ #1 \right\} }}
\newcommand\eps{\varepsilon}
\newcommand{\deq}{\ensuremath{ \stackrel{\textrm{def}}{=}}}
\newcommand{\HH}[3]{\ensuremath{H^{#1}\left(#2,#3\right)}}
\newcommand{\OO}{\ensuremath{\mathcal O}}
\newcommand\lra{\longrightarrow}
\newcommand{\equ}{\ensuremath{\,=\,}}

\newcommand{\dsubseteq}{\ensuremath{\,\subseteq\,}}

\newcommand{\vol}[2]{\ensuremath{{\rm vol}_{#1}\left( #2 \right) } }

\DeclareMathOperator{\Bbig}{Big}

\DeclareMathOperator{\Jac}{Jac}

\DeclareMathOperator{\Pic}{Pic}

\newenvironment{proofof}[1]{\trivlist\item[\hskip\labelsep{\it Proof of #1.}]}{\hspace*{\fill}$\Box$\endtrivlist}

\newcommand{\ie}{{\rm i.e.\ }}
\DeclareMathOperator{\reg}{reg}

\newcommand{\e}{\varepsilon}
\newcommand{\la}{\lambda}

\begin{document}

\title{Functions on Okounkov bodies coming from geometric valuations \\  \smallskip \smallskip \small with an appendix by S\'ebastien Boucksom}

\author{Alex K\"uronya, Catriona Maclean, Tomasz Szemberg}
\date{\today}
\maketitle
\thispagestyle{empty}

\begin{abstract}
   We study topological properties of  functions on Okounkov bodies as introduced by
   Boucksom and Chen \cite{BouChe11}, and Witt-Nystr\"om \cite{Nys09} in the case when they come from geometric valuations,
   and establish their continuity over the whole Okounkov body whenever the body is
   polyhedral. At the same time,  we exhibit an example that shows that continuity along the boundary does not hold in general.
   In addition, we study formal properties of such functions and the variation of their integrals in the N\'eron--Severi space.
   An appendix by S\'ebastien Boucksom adds a general subadditivity result.
\end{abstract}


\tableofcontents

\section{Introduction}

We aim  here   to study certain  functions on Newton--Okounkov bodies associated to Cartier divisors
which arise  from geometric valuations of the function field of the underlying variety.
We  investigate their formal properties, and show how to describe them explicitly in favourable cases by explicit computations using the geometry
of the underlying varieties..

Following the pioneering work of Okounkov \cite{Ok96}, Lazarsfeld--Musta\c t\u a \cite{LazMus09} and Kaveh--Khovanskii \cite{KavKh08} showed how to associate a
convex body to a big Cartier divisor $D$  via studying the vanishing behaviour of global sections along a complete flag of subvarieties. This body was then
called the Newton--Okounkov body of the divisor, and it soon proved to be a fundamental asymptotic invariant of $D$.  Subsequent applications of the theory of
Newton--Okounkov bodies (Okounkov bodies for short) outside complex geometry include connections to representation theory  \cite{Kav11} and Schubert calculus \cite{KST11}.

Okounkov bodies can be considered as generalizations of moment polytopes in symplectic geometry; on smooth toric varieties moment polytopes are special cases
of Newton--Okounkov bodies. Philosophically speaking,  Newton--Okounkov bodies replace the volume of a divisor $\vol{X}{D}$, which is just a number,
by a convex body, thus providing it with extra structure.  Arguably the most interesting application of this theory so far is related
to the moment polytope point of view: in a recent seminal paper, Harada and Kaveh \cite{HK12} construct completely integrable systems on
certain smooth projective varieties that map onto certain Okounkov bodies.

Coming from ideas in complex analytic geometry, Witt-Nystr\"om \cite{Nys09} and Bouck\-som--Chen \cite{BouChe11} present ways to obtain continuous functions
on Okounkov bodies given a multiplicative filtration of the associated section ring. As explained by Witt-Nystr\"om in \cite{Nys10}, some of these functions are
closely related to Donaldson's test configurations \cite{Don1,Don2,RT06,Szek11} and K-stability.

In this paper we consider functions arising from filtrations that carry significant geometric information, more
specifically, we look at filtrations coming from geometric valuations of the function field.
As a rough approximation, the value of a function associated to a valuation $\nu$ at a point of the Okounkov
body is the supremum over the values of $\nu$ at sections with the same vanishing vector. The most important property of the functions associated to filtrations is that their image
measure describes the asymptotic behaviour of the jumping values of the filtration.

To be more specific, given a geometric valuation $\nu$ on a projective variety $X$ over the complex number field, we obtain a filtration $\fbul \C(X)$ of the
function field $\C(X)$ by setting
\[
 \calf_t \C(X) \deq \left\{f\in \C(X):\; \nu(f)\geq t\right\}\  \ \ \text{for $t\in\R$.}
\]
For a big Cartier divisor $D$ on $X$, this induces a multiplicative filtration on the section ring $R(X,D)=\oplus_{m=0}^{\infty}\HH{0}{X}{\OO_X(mD)}$.
This filtration has at most linear growth. By the method  of Boucksom and Chen \cite{BouChe11} or Witt-Nystr\"om \cite{Nys09}, $\fbul$ then gives rise to a
non-negative concave function
\[
 \phi_{\fbul} : \Delta_{\ybul}(D) \lra \R \ ,
\]
which we refer to as an Okounkov function on $\Delta_{\ybul}(D)$.

By concavity, these functions are always continuous  in the interior of the underlying Okounkov bodies, nevertheless, since continuous function on compact
spaces have particularly good properties, it is important to be able to control their behaviour on the boundary. Our main result concerns exactly this
question.

\begin{theorem}\label{thm:cont}
\begin{enumerate}
 \item Let $X$ be a projective variety, $Y_\bullet$ and admissible flag, $D$ a $\Q$-effective Cartier divisor on $X$, $V_\bullet$ a graded linear series
associated to $D$. Pick a geometric valuation $\nu$ of $\C(X)$. If the Newton--Okounkov body $\Delta_{\ybul}(\vbul)$ is a polytope (not necessarily rational),
then the Okounkov function $\phi_\nu:\Delta_{\ybul}(\vbul)\to\R$ is continuous
on the whole $\Delta_{\ybul}(\vbul)$.
\item On the other hand, there exists  a variety $X$, equipped with a flag $Y_1\ldots, Y_n$, and a geometric valuation on $X$, $\nu$, and an ample divisor $D$
on $X$
such that the Okounkov function $\phi_{\nu}$ on $\Delta(D)$ is not continuous.
\end{enumerate}
\end{theorem}
\noindent
The Theorem will be proven in subsections 4.1 and 4.3. Coupled with the fact that on surfaces Okounkov bodies of divisors are polygones \cite{KLM},
we obtain the  following.

\begin{corollary}
 Let $X$ be a smooth projective surface over the complex numbers, $\ybul$ an admissible flag, $L$ a big Cartier divisor,
 $\nu$ a geometric valuation on $X$. Then the function $\phi_\nu: \Delta_{\ybul}(L)\to \R$  is continuous.
\end{corollary}

The functions $\phi_{\fbul}$ have many interesting formal properties, one of them is an interesting  reduction property.
More precisely,  we observe  that given a judicious choice of a flag, the computation of $\phifbul$ can be reduced to the boundary of the Newton--Okounkov
body.

\begin{theorem}
  Assume that $\vbul$ is a graded linear series associated to the line bundle
   $L$ such that there is an irreducible divisor $Y_1\in|L|$. We take a flag $\ybul$
   whose divisorial part is $Y_1$. Let $\fbul$ be a filtration on $\vbul$
   defined by a geometric valuation $\nu$. Then for $x=(x_1,\dots,x_n)\in\Delta_{\ybul}(\vbul)$
   we have
\[
   \phifbul(x_1,\dots,x_n) \equ (1-x_1)\phifbul\left(0,\frac{x_2}{1-x_1},\dots,\frac{x_n}{1-x_n}\right)
      + x_1\cdot\nu(Y_1).
\]
\end{theorem}
\noindent
We verify this claim in  Theorem~\ref{thm:function from boundary} below.

The Okounkov functions we define are without exception integrable. Boucksom and Chen show  in passing that the integral of Okounkov functions is independent of the
choice of the flag.  In a sequel \cite{KMS} to our current paper we establish that the maximum of an Okounkov function is independent of the chosen
flag as well.

The integrals $I(D;\nu)$ give rise to new  invariants of divisors or graded linear series. For functions associated to test configurations Witt-Nystr\"om
observes in \cite{Nys10} that their normalized integral equals the Futaki invariant $F_0$, nevertheless, the geometric meaning of $I(D;\nu)$ is quite unclear.

Let $\nu$ be a geometric valuation, $D$ a Cartier divisor on $X$. We define
\[
 I(D;\nu) \deq \frac{1}{\vol{X}{D}}\int_{\Delta_{\ybul}(D)} \phi_\nu\
\]
for an arbitrary admissible flag $\ybul$ on $X$. Then one can interpret \cite[Theorem 1.11]{BouChe11} as saying that  $I(D;\nu)$ is the limit of
normalized sums of jumping values of the underlying filtration.

We summarize fundamental properties of $I(D;\nu)$ in the following statement.

\begin{theorem}
With notation as above, the invariant $I_\nu$ has the following properties.
\begin{enumerate}
 \item If $D\equiv D'$, then $I_\nu(D)=I_\nu(D')$.
\item  For a positive integer $a$, one has $I_\nu(aD) = a\cdot I_\nu(D)$.
\item There is a unique extension of $I_\nu$ to a continuous function
\[
I_\nu:  \Bbig(X) \lra \R_{\geq 0}\ .
\]
\end{enumerate}
\end{theorem}
\noindent
The claims above will be shown in Proposition~\ref{prop:num equiv}, Remark~\ref{rem:homogenity of I(D nu)}, and Proposition~\ref{prop:integral cont}.

A few words about the organization of this paper: we start, in Section 2, by recalling the definitions of Okounkov bodies, giving some examples of
calculations, and proving some technical lemmas which will be needed in the rest of the paper. Section 3 contains definitions and technical preliminaries on
filtrations of algebras. In Section 4, we then present Witt-Nystr\"om and Boucksom-Chen's definitions of Okounkov functions, deal with the important issue
of continuity, and  calculate several explicit examples of Okounkov functions before turning to the question of invariants of Okounkov functions.
We treat integrals of Okounkov functions in Section 5.
One of essential tools used repeatedly in the present paper is Fekete Lemma \cite{Fek23}.
Section 6 is an appendix by S\'ebastien Boucksom presenting a general Fekete-type lemma originating from \cite{Nys09}, and which can be used to construct the Okounkov function of a filtration.

\paragraph{Acknowledgments.} We first  heard about the possibility of defining interesting functions on Okounkov bodies from Bo Berndtsson at a workshop in Oberwolfach. We are grateful
to S\'ebastien Boucksom, Lawrence Ein, Patrick Graf, Daniel Greb, and Rob Lazarsfeld for helpful discussions.

During this project Alex K\"uronya  was supported in part by  the DFG-Forscher\-grup\-pe 790 ``Classification of Algebraic Surfaces and Compact Complex
Manifolds'', and the OTKA Grants 77476 and  81203 by the Hungarian Academy of Sciences. Tomasz Szemberg's research was partly supported by NCN grant
UMO-2011/01/B/ ST1/04875. Part of this work was done while the second author was visiting the Uniwersytet Pedagogiczny in Cracow. We would like to use this
opportunity to thank the Uniwersytet Pedagogiczny for the excellent working conditions.

\section{Definitions and examples}
\subsection{Okounkov bodies}
   We recall here some basic notions and properties of Okounkov bodies and establish  notation.
   A systematic development of the theory in the geometric setting  has been initiated in \cite{LazMus09} and \cite{KavKh08}, we refer
   to these  articles for motivation and additional details. The phrases 'Okounkov body' and 'Newton--Okounkov body' will
   be used interchangebly throughout the text.

   Let $X$ be an irreducible projective variety of dimension $n$ and
   $$\ybul:\;X=Y_0\supset Y_1\supset\dots\supset Y_{n-1}\supset Y_n=\left\{p\right\}$$
   be a flag of irreducible subvarieties of $X$ such that
   $\codim_X(Y_i)=i$ and $p$ is a smooth point of each $Y_i$.

   Let $D$ be a Cartier divisor on $X$ and let $\vbul$ be a graded linear series
   associated to $D$ (see \cite[Definition 2.4.1]{PAG}).

   The flag $\ybul$ defines a rank-$n$ valuation
   $$\nuy:V_k\setminus\st{0}\to\Z^n $$
   in the following way. Given a section $0\neq s\in V_k\subset H^0(X, kD)$ we set
   $$\nu_1=(\nuy)_1(s):=\ord_{Y_1}(s).$$
   This determines a section
   \begin{equation}\label{eq:s1}
      \wtilde{s}\in H^0(X,kD-\nu_1Y_1),
   \end{equation}
   which does not vanish identically along $Y_1$, and thus restricts
   to a non-zero section
   $$s_1\in H^0(Y_1,(D-\nu_1Y_1)\restr{Y_1}).$$
   We repeat the above construction for $s_1$ and so on. In this way
   we produce a valuation vector
   $$\nuy(s)=\left((\nuy)_1(s),\dots,(\nuy)_n(s)\right)\in\Z^n$$
   and an element
   \begin{equation}\label{eq:graded semigroup}
      (\nuy(s),k)\in\Gamma_{\ybul}(\vbul)\subset\Z^{n+1}
   \end{equation}
   in the \emph{graded semigroup} of the linear series $\vbul$.
   Let $\valybul(\vbul)\subset\R^n$ be the set of all normalized valuation vectors obtained as above i.e.
   $$\valybul(\vbul)=\left\{\frac1k\nuy(s):\;s\in V_k,\,k=1,2,3,\dots\right\} \subseteq \R^n\,.$$
   We write simply $\valybul(D)$ if $\vbul$ is the complete linear series of $D$.
   For a given element $v\in\Val_{Y_\bullet}(V_\bullet)$, we define
   $$
      S_v \deq \st{k\in\N\mid \exists s\in V_k\colon \nu_{Y_\bullet}(s)=kv}\ .
   $$
   Clearly $S_v$ is an additive subsemigroup in $\N$.
\begin{definition}[Okounkov body of a graded linear series]
   The \emph{Okounkov body} $\dybul(\vbul)$ of $\vbul$ is the closed convex hull of the set $\valybul(\vbul)$.
\end{definition}

   \begin{remark}
   Note that we abuse notation slightly, since $\dybul(\vbul)$ is in general a convex compact set only. The above definition of a Newton--Okounkov body
   works fine for any arbitrary graded linear series $\vbul$. In fact, an interesting topic of Okounkov bodies for non-big pseudo-effective
   divisors has been taken on recently by Di Biagio and Pazienza \cite{BiaPaz12}.

    As explained in \cite[Lemma 2.6]{LazMus09}, if $\vbul$ is big, then the corresponding Okounkov body will indeed
   contain an open ball. By a big graded linear series we mean one satisfying Condition (C) of \cite[Definition 2.9]{LazMus09}.
   \end{remark}

  We will see below that in fact taking the closure is enough as the normalized
   valuation vectors are dense in the convex hull.
   Again, if $\vbul$ is the complete linear series associated to a Cartier divisor $D$ on $X$, then we write $\dybul(D)$ for its Okounkov body.

\begin{example}[Okounkov bodies of $\P^2$ and its blow up]\label{ex:ok of p2 and blow up}\rm
   Let $\ell$ be a line in $X_0=\P^2$ and $P_0\in\ell$ a point.
   In  what follows we operate with a fixed flag
   $$\ybul:\; X_0\supset\ell\supset\left\{P_0\right\}.$$\\
   a). Let $D_0=\calo_{\P^2}(2)$. Then $\dybul(D_0)$ is twice the standard simplex in $\R^2$
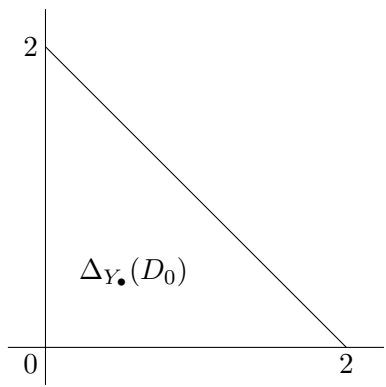
\unitlength.1mm
\begin{center}
\begin{pgfpicture}{0cm}{0cm}{6cm}{6cm}
   \pgfline{\pgfxy(0.5,1)}{\pgfxy(5.5,1)}
   \pgfline{\pgfxy(1,0.5)}{\pgfxy(1,5.5)}
   \pgfline{\pgfxy(5,1)}{\pgfxy(1,5)}
   \pgfputat{\pgfxy(5,0.9)}{\pgfbox[center,top]{2}}
   \pgfputat{\pgfxy(0.9,0.9)}{\pgfbox[right,top]{0}}
   \pgfputat{\pgfxy(0.9,5)}{\pgfbox[right,center]{2}}
   \pgfputat{\pgfxy(2.9,2)}{\pgfbox[right,center]{$\dybul(D_0)$}}
\end{pgfpicture}
\end{center}
   b). Let $P_1$ be a point in the plane not lying on the line $\ell$
   and let $f_1:X_1=\Bl_{P_1}X_0\to X_0$ be the blow up of $P_1$
   with exceptional divisor $E_1$. For $D_1=f_1^*\calo_{\P^2}(2)-E_1$, we have
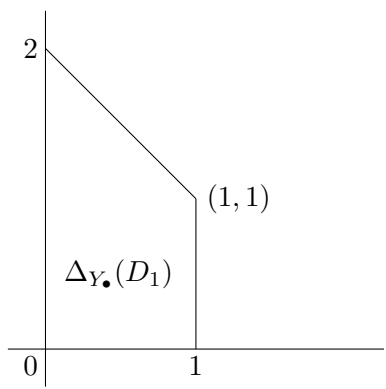
\begin{center}
\begin{pgfpicture}{0cm}{0cm}{6cm}{6cm}
   \pgfline{\pgfxy(0.5,1)}{\pgfxy(5.5,1)}
   \pgfline{\pgfxy(1,0.5)}{\pgfxy(1,5.5)}
   \pgfline{\pgfxy(1,5)}{\pgfxy(3,3)}
   \pgfline{\pgfxy(3,1)}{\pgfxy(3,3)}
   \pgfputat{\pgfxy(3,0.9)}{\pgfbox[center,top]{1}}
   \pgfputat{\pgfxy(0.9,0.9)}{\pgfbox[right,top]{0}}
   \pgfputat{\pgfxy(0.9,5)}{\pgfbox[right,center]{2}}
   \pgfputat{\pgfxy(2.7,2)}{\pgfbox[right,center]{$\dybul(D_1)$}}
   \pgfputat{\pgfxy(4,3)}{\pgfbox[right,center]{$(1,1)$}}
\end{pgfpicture}
\end{center}
   c). Let $P_1, P_2$ be points in the plane not lying on the line $\ell$
   and such that $P_0,P_1,P_2$ are not collinear. Let $f_2:X_2=\Bl_{P_1,P_2}X_0\to X_0$
   be the blow up of $P_1, P_2$ with exceptional divisors $E_1, E_2$.
   For a big and nef line bundle $D_2=f_2^*\calo_{\P^2}(2)-E_1-E_2$, we have
   then the Okounkov body as in the picture c1) below. The picture c2)
   shows the Okounkov body of the same line bundle under assumption
   that $P_0,P_1,P_2$ are collinear.
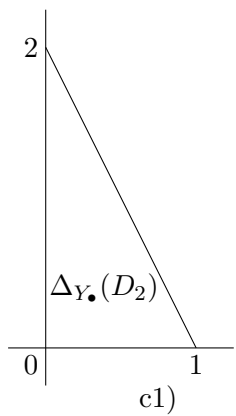
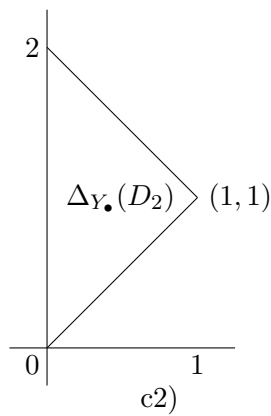
\begin{center}
\begin{pgfpicture}{0cm}{0cm}{12cm}{6cm}
   \pgfline{\pgfxy(0.5,1)}{\pgfxy(3.5,1)}
   \pgfline{\pgfxy(1,0.5)}{\pgfxy(1,5.5)}
   \pgfline{\pgfxy(3,1)}{\pgfxy(1,5)}
   \pgfputat{\pgfxy(3,0.9)}{\pgfbox[center,top]{1}}
   \pgfputat{\pgfxy(0.9,0.9)}{\pgfbox[right,top]{0}}
   \pgfputat{\pgfxy(0.9,5)}{\pgfbox[right,center]{2}}
   \pgfputat{\pgfxy(2.5,1.8)}{\pgfbox[right,center]{$\dybul(D_2)$}}
   \pgfputat{\pgfxy(2.5,0.5)}{\pgfbox[center,top]{c1)}}
   \pgfline{\pgfxy(6.5,1)}{\pgfxy(9.5,1)}
   \pgfline{\pgfxy(7,0.5)}{\pgfxy(7,5.5)}
   \pgfline{\pgfxy(7,5)}{\pgfxy(9,3)}
   \pgfline{\pgfxy(7,1)}{\pgfxy(9,3)}
   \pgfputat{\pgfxy(9,0.9)}{\pgfbox[center,top]{1}}
   \pgfputat{\pgfxy(6.9,0.9)}{\pgfbox[right,top]{0}}
   \pgfputat{\pgfxy(6.9,5)}{\pgfbox[right,center]{2}}
   \pgfputat{\pgfxy(8.7,3)}{\pgfbox[right,center]{$\dybul(D_2)$}}
   \pgfputat{\pgfxy(10,3)}{\pgfbox[right,center]{$(1,1)$}}
   \pgfputat{\pgfxy(8.5,0.5)}{\pgfbox[center,top]{c2)}}
\end{pgfpicture}
\end{center}
\end{example}

\subsection{Density of valuation vectors}
   Here we verify that the points
   in $\valybul(\vbul)$ are dense in the convex hull of $\valybul(\vbul)$, hence also in $\dybul(\vbul)$. This means in particular that the
   closure of $\valybul(D)$ is convex.

   We first treat the case of a complete linear series $|L|$
   on a curve $C$, because it is particularly transparent and
   constructive.

   Fix a flag $\ybul:\, C=Y_0\supset Y_1=\left\{p\right\}$, and recall that $\dybul(L)=[0,\deg L]$, see \cite[Example 1.3]{LazMus09}.
   For a given point $q\in C$ (which might or might not be equal to $p$), we write
   \[
     \shs_{v,k}(q) \deq \st{t\in\R\mid \exists s\in V_k\colon \ord_q(s)\geq t\, ,\, \nu_{Y_\bullet}(s)=kv}\ .
   \]
   By definition $\shs_{v,k}(q)\neq \emptyset$ if and only if $k\in S_v$.

\begin{lemma}[Complete linear series on a curve]\label{lemma:val vector ample}
 With notation as above we have the following claims.
\begin{enumerate}
    \item[(1.)] $\Val_{Y_\bullet}(L)\setminus\deg L \equ  [0,\deg L)\cap \Q$, in particular, the set of normalized vanishing vectors is dense in
       $\Delta_{Y_\bullet}(L)$.
    \item[(2.)] For $v\in\valybul(\vbul)$ the set $S_v\subseteq \N$ is an additive subsemigroup with the exponent $e(S_v)=d$, where $d$ equals the denominator
    of the rational number $v$ in its reduced form if $v<\deg L$, and $d$ is the order of $L-(\deg L)p$ in $\Pic^0$ otherwise.
    \item[(3.)] For given $v\in\Val_{Y_\bullet}(L)$ and $q\in C$, the  sequence
    \[
       a_k \deq  \frac{1}{dk} \sup\shs_{v,dk}(q)
    \]
is convergent.
\end{enumerate}
\end{lemma}

\begin{remark}\label{rem:deg L}
Let us discuss the possibility of $\deg L\in\Val_{\ybul}(L)$. By definition, this happens precisely if $\HH{0}{C}{\OO_C(mL-m(\deg L)p)}\neq 0$ for some $m\geq
0$. This is equivalent to asking that
\[
 L -(\deg L)p
\]
   is a torsion point in $\Jac(C)$. This is certainly not the case for most
   line bundles $L$ on a non-rational curve $C$.
\end{remark}

\proofof{Lemma \ref{lemma:val vector ample}}
(1) By construction all elements $v$ in $\Val_{\ybul}(L)$ are rational numbers,
   and they sit inside $\Delta_{\ybul}(L)=[0,\deg L]$, in particular, $v\leq \deg L$.
   \\
   In the other
direction, let $v\in \Q\cap [0,\deg L)$.
   Let $m=kd$ be so large that
   \begin{equation}\label{eq:h1vanish}
      h^1(C,\calo_C(mL-mv\cdot p))=0\; \mbox{ and }\; h^1(C,\calo_C(mL-(mv+1)\cdot p))=0.
   \end{equation}
   We want to show that $m\in S_v$, i.e. that there exists
   a section of $\calo_C(mL)$ vanishing at $p$ to order exactly $mv$.
   The vanishing in \eqnref{eq:h1vanish} implies
   \begin{equation}\label{eq:h0differ}
      h^0(C,\calo_C(mL-mv\cdot p)) > h^0(C,\calo_C(mL-(mv+1)\cdot p))
   \end{equation}
   via Riemann--Roch applied on $C$ to both systems. They are non-empty
   by the same token. It follows that there is a section in $mL$
   whose vanishing order at $p$ is exactly $mv$. Hence $m\in S_v$.

(2) The claim that $S_v$ is an additive  subsemigroup of $\N$ is a consequence of the fact that $\nu_{\ybul}$ behaves logarithm--like on global sections.
   It must be $d|e(S_v)$, since
   $mv$ is an integer for every $m\in S_v$.
   In order to show the equality, we need to check that
   $S_v$ contains all natural numbers $kd$ for  $k\gg 0$.

   This follows again from a Riemann--Roch computation.
   Let $v\in \valybul(L)$ be fixed with $v<\deg(L)$. Since $L-vp$ is an ample $\Q$-divisor,
   there exists then $m_0$ such that for all $m\geq m_0$
   one has the vanishing \eqnref{eq:h1vanish} whenever $mv$ is an integer.

   Let $k$ be so that $kd$ is an integer satisfying $kd>m_0$. Then Riemann-Roch
   together with the vanishing implies as above
   \begin{equation}\label{eq:h0differ1}
      h^0(C,\calo_C(kdL-kdv\cdot p)) > h^0(C,\calo_C(kdL-(kdv+1)\cdot p)),
   \end{equation}
   which in turn means that $kd\in S_v$.

   The case of $v=\deg L$ is immediate from Remark~\ref{rem:deg L}.

(3) Part (2) implies that $\shs_{v,dk}(q)\neq \emptyset $ for $k\gg 0$, hence $b_{dk}:=\sup\shs_{v,dk}(q)$ forms a super-additive sequence of rational numbers (that is,
different from $-\infty$) in  $k$. Consequently, the limit of the sequence $a_k:=\frac{1}{dk}b_{dk}$ exists by \cite{Fek23}. 
\endproofof

We now move on to the general case when the underlying variety $X$ is allowed to have arbitrary dimension, and $V_\bullet$ is a graded linear series.

\begin{lemma}\label{lem:density of valuation}
   Let $X$ be a projective variety, $V_\bullet$ a graded linear series
   (not necessarily big) associated to a $\Q$-effective Cartier divisor $D$. Then
\begin{enumerate}
 \item[(1.)] The set $\Val_{Y_\bullet}(V_\bullet)$ is dense in
$\Delta_{Y_\bullet}(V_\bullet)$.
\item[(2.)] For $v\in\valybul(\vbul)$ the set $S_v\subseteq \N$ is an additive subsemigroup.
\item[(3.)] For given $v\in\Val_{Y_\bullet}(V_\bullet)$ and $q\in C$, the  sequence
\[
a_k \deq  \frac{1}{k} \sup\shs_{v,k}(q)
\]
with $k$ running through the elements of $S_v$ is convergent.
\end{enumerate}
\end{lemma}
\begin{proof}
   (1.) The argument now is less constructive than in the case of curves, on the other hand it explains
   why the closure of the set of normalized valuation vectors is a convex set.
   Let
   $v_1,v_2\in\Val_{Y_\bullet}(V_\bullet)$, and let $m_i\in\N$, $s_i\in V_{m_i}\subseteq \HH{0}{X}{\OO_X(m_iD)}$ for $i=1,2$ be such that
\[
 \nu_{\ybul}(s_i) \equ m_iv_i\ \ \ \text{for $i=1,2$.}
\]
Then $s_1^{m_2}s_2^{m_1}\in V_{2m_1m_2}$, and
\[
 \nu_{\ybul}(s_1^{m_2}s_2^{m_1}) \equ m_2\cdot\nu_{\ybul}(s_1) + m_1\cdot\nu_{\ybul}(s_2) \equ m_2m_1v_1+m_1m_2v_2 \equ m_1m_2(v_1+v_2)\ ,
\]
hence
\[
\frac{1}{2m_1m_2} \Gamma_{\ybul}(V_{2m_1m_2})  \ni \frac{1}{2m_1m_2}\cdot \nu_{\ybul}(s_1^{m_2}s_2^{m_1}) \equ \frac{1}{2}(v_1+v_2)\ .
\]
   This shows that the midpoint between two normalized valuation vectors is again a normalized valuation vector, hence density follows.

The above  argument shows also that for $v_1,v_2\in\valybul(\vbul)$ the segment $\overline{v_1v_2}$
   is contained in the closure $\Delta_{\ybul}(\vbul)$, therefore the  closure is a convex set.

   (2.) The fact that $S_v$ is an additive subsemigroup follows from the valuation-like behavior of $\nu_{\ybul}$
   and the property that $V_k\cdot V_m\subseteq V_{k+m}$.

   (3.) The proof is the same as in the case of curves.
\end{proof}

\begin{remark}\rm
   Note that the property (1.) in Lemma \ref{lem:density of valuation}
   has been silently used in the proof of \cite[Proposition 2.1]{LazMus09}.
   We include a proof here for the lack of a direct reference.
\end{remark}

\section{Filtrations}

Filtrations of vector spaces and graded algebras are used by Boucksom and Chen \cite{BouChe11} to define functions on Okounkov bodies.
Here we recall the notions we will need, and look at situations that are interesting from the geometric point of view.
The formal considerations come from \cite{BouChe11} for the most part.

\subsection{Filtrations on vector spaces}
   We begin by making it precise what we mean by a filtration in this article.
\begin{definition}[Filtration]
   Let $E$ be a finite dimensional complex vector space. We call a  family
   $\fbul E$ of linear subspaces  of $E$ indexed by real numbers $t\in\R$ a \emph{filtration} on $E$ if
   \begin{enumerate}
   \item for all real numbers $t\in\R$, $\calf_tE\subset E$ is a vector subspace;
   \item $\fbul$ is non-increasing i.e.
   $$\mbox{from } t_1\leq t_2\;\mbox{ follows }\; \calf_{t_1}E\supset\calf_{t_2}E;$$
   \item $\fbul$ is left continuous i.e.
   $$\lim\limits_{t\rightarrow t_0^-}\calf_tE=\calf_{t_0}E;$$
   \item $\fbul$ is left and right bounded i.e. there exist real numbers $t_l$ and $t_r$ such that
   $$\calf_{t_l}E=E\;\mbox{ and }\; \calf_{t_r}E=0.$$
   \end{enumerate}
\end{definition}
   A standard situation for this article is the following.
\begin{example}[Filtration defined by a valuation]\label{ex:filt defined by val}
   Let $X$ be an irreducible projective variety and let $E\subset\C(X)$ be
   a finite dimensional complex vector subspace of the function field of $X$.
   Let $\nu:\C(X)\to\Z$ be a rank $1$ valuation. Then
   $$\calf_tE:=\left\{f\in E:\; \nu(f)\geq t\right\}$$
   is a filtration on $E$.

   The sort of valuation we are mostly interested are geometric valuations, that is, orders of vanish along a subvariety.
\end{example}
   Given a filtration we define jumping numbers.
\begin{definition}[Jumping numbers]
   Let $\fbul$ be a filtration on a finite dimensional vector space $E$.
   The numbers
   $$e_j(E,\fbul):=\sup\left\{t\in\R:\;\dim\calf_tE\geq j\right\}$$
   for $j=1,\dots,\dim E$ are the \emph{jumping numbers} of the filtration $\fbul$.
   We suppress $E$ and $\fbul$ if the vector space and the filtration are clear from the context
   and write simply $e_j$ in such a case.
\end{definition}
   Note that we have the following monotonicity
   $$e_{\min}(E,\fbul):=e_{\dim E}(E,\fbul)\leq\cdots\leq e_1(E,\fbul)=:e_{\max}(E,\fbul).$$
   In particular,
   $$e_{\min}(E,\fbul)=\inf\left\{t\in\R:\;\calf_tE\neq E\right\}\;\mbox{ and }\;
     e_{\max}(E,\fbul)=\sup\left\{t\in\R:\;\calf_tE\neq 0\right\}.$$
   Following Boucksom and Chen, we define the \emph{mass} of $(E,\fbul)$ as
   $$\mass(E,\fbul):=\sum\limits_{j=1}^{\dim E} e_j(E,\fbul).$$

\begin{remark}
 Once the functions associated to filtrations will have been defined, the mass of a filtration will be related to the integral of the corresponding function
over Newton--Okounkov bodies.
\end{remark}

\begin{example}[Jumping numbers on homogeneous polynomials]
   Let $X=\P^2$ and $E=H^0(\calo_{\P^2}(1))$. We consider the filtration $\fbul$
   on $E$ introduced by a geometric valuation $\nu$ given
   by the order of vanishing $\ord_p$ at a fixed point $p\in\P^2$ as in Example \ref{ex:filt defined by val}.
   Then
   $$e_{\min}=e_3=0,\; e_2=e_1=e_{\max}=1\;\mbox{ and }\; \mass=2.$$
\end{example}


\subsection{Filtrations on graded algebras}
   The constructions from the previous part extend to the setting of graded $\C$-algebras.

\begin{definition}[A filtration on a graded object]\label{def:graded filtration}
   Let
   $$\ebul=\bigoplus_{k\geq 0}E_k$$
   be a graded $\C$-algebra with $E_0=\C$ and $\dim E_k$ finite for all $k$.
   A family $\fbul\ebul$ of subspaces of $\ebul$ is a \emph{filtration of the graded algebra} $\ebul$
   if $\fbul E_k$ is a filtration on the vector space $E_k$ for all $k$.\\
   We say that $\fbul$ is \emph{multiplicative} if for all $s,t\in\R$
   and all $m,n$ we have
   $$(\calf_t E_m)\cdot (\calf_s E_n)\subset \calf_{t+s} E_{m+n}.$$
\end{definition}
\begin{example}[A filtration given by a valuation]\label{ex:graded filt by valuation}
   Let $X$ be an irreducible projective variety.
   Let $\ebul=\oplus_{k\geq 0}E_kT^k\subset\C(X)[T]$ be a graded subalgebra which is connected
   (i.e. $E_0=\C$) and locally finite (that is,  $\dim E_k<\infty$ for all $k$).

   Let $\nu$ be a geometric valuation on $\C(X)$ i.e. a valuation defined
   by the order of vanishing along a  subscheme $Z$ in $X$. Since
   $$ \nu(f_1\cdot f_2) \equ \nu(f_1)+\nu(f_2) \ ,$$
   the expression
   $$\calf_tE_k=\left\{f\in E_k:\; \nu(f)\geq t\right\}$$
   defines a multiplicative filtration.
\end{example}
\begin{definition}[Linearly bounded filtrations]
   In the setup of Definition \ref{def:graded filtration}, we say that
   the filtration $\fbul\ebul$ is \emph{linearly left bounded},
   if there exists a constant $C>0$ such that for all $k$ we have
   $$e_{\min}(E_k,\fbul)\geq -C\cdot k.$$
   Similarly, $\fbul\ebul$ is \emph{linearly right bounded},
   if
   $$e_{\max}(E_k,\fbul)\leq C\cdot k$$
   for a fixed constant $C>0$ and all $k$.
\end{definition}
   We can generalize jumping numbers to the graded setting.
\begin{definition}[Asymptotic jumping numbers]
   With notation as in Definition \ref{def:graded filtration} we set
   $$e_{\min}(\ebul,\fbul):=\liminf \frac1k e_{\min}(E_k,\fbul)\;\mbox{ and }\;
     e_{\max}(\ebul,\fbul):=\limsup \frac1k e_{\max}(E_k,\fbul).$$
\end{definition}

   Note that a  filtration $\fbul\ebul$  of the graded $\C$-algebra $\ebul$ is linearly left bounded
   if and only if $e_{\min}(\ebul,\fbul)>-\infty$ and similarly, it is linearly
   right bounded if and only if $e_{\max}(\ebul,\fbul)<\infty$

\begin{proposition}[Filtration on a graded linear series]\label{prop:filt on a graded lin series}
   Let $X$ be an irreducible normal projective variety of dimension $n$, $D$ a Cartier
   divisor on $X$ and $\vbul$ a graded linear series
   defined by $D$.
   Furthermore let $Z$ be a subvariety in $X$, $\nu=\ord_Z$ be the geometric
   valuation defined by $Z$, and let $\fbul\vbul$ be the filtration given by $\nu$
   as in Example~\ref{ex:graded filt by valuation}. Then $\fbul$ is linearly
   left and right bounded.
\end{proposition}
\proof
   The valuation $\ord_Z$ is left bounded as $\ord_Z(s)\geq 0$ for all $s\neq 0$, hence also
   $$e_{\min}(V_k,\fbul)\geq 0 \;\mbox{ for all }\; k.$$
   For the right boundedness we claim that there exists a positive constant $C$
   such that
   $$\max\left\{\ord_Z(s):\; s\in V_k\right\}\leq C\cdot k$$
   for all $k$. It is enough to prove this claim for the complete linear series
   $V_k=H^0(X,kD)$. To this end let $\pi:Y\to X$ be the blowing up along $Z$.
   There exists a unique irreducible component $E$ of the exceptional locus
   of $\pi$ mapping surjectively onto $Z$. For this component we have
   $$\ord_Z(s)=\ord_E(\pi^*s)\;\mbox{ for all }\; s\in H^0(X,kD).$$
   Let $H$ be an ample line bundle on $Y$. There exists $C>0$ such that
   $$(\pi^*D-C\,E)\cdot H^{n-1}<0.$$
   This implies that $\ord_Z(s)=\ord_E(\pi^*s)\leq C\cdot k$ for all $s\in H^0(X,kD)$.
   Thus we have
   $$e_{\max}(H^0(X,kD),\fbul)=\max\left\{\ord_Zs:\; s\in H^0(X,kD)\right\}\leq C\cdot k.$$
\endproof

\begin{example}[Asymptotic order of vanishing]\label{ex:asympt order of vanish}
   Let $X$ be a normal projective variety and $\vbul$ a graded linear series on $X$.
   For a geometric valuation $\nu$ we define a filtration $\fbul$ on $\vbul$
   as in Proposition \ref{prop:filt on a graded lin series} and we set
   $$\nu(V_k):=\min\left\{\nu(s):\; s\in V_k\setminus\left\{0\right\}\right\}.$$
   Then
   $$\emin(V_k,\fbul)=\nu(V_k)$$
   and
   $$\emin(\vbul,\fbul)=\lim\frac1k \nu(V_k)=\inf\frac1k \nu(V_k)$$
   recovers  the asymptotic order of vanishing along the center of $\nu$ as defined in \cite[Definition 2.2]{AIBL}.
   The fact that we can write $\inf$ and $\lim$ instead of $\limsup$
   is accounted for by the subadditivity of the sequence $\nu(V_k)$:
   $$\nu(V_{k+m})\leq\nu(V_k)+\nu(V_m)$$
   as explained in \cite[Lemma 2.1]{AIBL} and  Fekete's Lemma \cite{Fek23}.

\end{example}
   The number $\emax$ behaves similarly under mild additional assumption.
\begin{lemma}[$\emax$ for graded linear series]\label{lem:emax for graded lin ser}
   Let $\vbul$ be a graded linear series such that $V_k\neq 0$ for all $k$.
   Then
   $$\emax(\vbul,\fbul)=\lim\limits_{k}\frac1k\emax(V_k,\fbul)=\sup_k\frac1k\emax(V_k,\fbul)$$
   for an arbitrary filtration on $\vbul$.
\end{lemma}
\proof
   This follows by the superadditivity of the sequence $\left\{\emax(V_k,\fbul)\right\}$
   and again  Fekete's Lemma, see also \cite[Lemma 1.4]{BouChe11}.
\endproof
\begin{corollary}[Jumping numbers of Veronese algebra]
    Let $X$ be a normal
   projective variety and $\vbul$ a graded linear series. Fixing
   a positive integer $m$, the
   Veronese algebra $V_{m\bullet}$
   is a graded linear series as well. For a filtration
   $\fbul$ defined on $\vbul$ by a geometric valuation $\mu$ on $X$
   and the corresponding filtration $\calf_{m\bullet}$ on $V_{m\bullet}$, we have
   $$\emin(V_{m\bullet},\calf_{m\bullet})=m\emin(\vbul,\fbul)\;\mbox{ and }\;
     \emax(V_{m\bullet},\calf_{m\bullet})=m\emax(\vbul,\fbul).$$
\end{corollary}
\begin{proof}
 It follows from Example \ref{ex:asympt order of vanish} and Lemma \ref{lem:emax for graded lin ser}
   that $\emin$ and $\emax$ scale well for graded subalgebras.
\end{proof}

   We get the following characterization of the maximal jumping
   number in case of a complete linear series.
\begin{remark}[Maximal jumping number of a complete linear series]\label{rem:max jumping number}
   Let $X$ be a normal projective variety, $Z$ an irreducible smooth
   subvariety of $X$. Let $D$ be a Cartier divisor on $X$
   and $\vbul=R(X,D)=\oplus_{k\geq 0} H^0(X,kD)$ be the section ring of $D$.
   Moreover let $\pi:Y\to X$ be the normalized blowing up of $Z$ with
   the exceptional divisor $E$.
   Then for $s\in H^0(X,kD)$ we have
   $$\ord_Z(s)=\ord_E(\pi^*s)=\max\left\{m\in\N:\; \div(\pi^*s)-mE\;\mbox{is effective}\right\}.$$
   Let $\fbul$ be the filtration on $\vbul$ induced by the order of vanishing along $Z$.
   Then it follows from Example \ref{ex:asympt order of vanish} that
   $$\emax(R(X,D),\fbul)=\sup \frac1k\max\left\{\ord_Z(s):\; s\in H^0(X,kD)\right\}=$$$$
      \sup\left\{t\geq 0:\; \pi^*L-tE\;\mbox{is big}\right\}=:\mu_E(\pi^*D)=:\mu(D,Z).$$
   Thus we see that $\emax$ is in this situation closely related to the geometry
   of the big cone on $Y$. Namely, it is the non-negative value of $t$ at which the ray
   $\pi^*(L)-tE$ intersects the boundary of the big cone.
\end{remark}


\section{Functions on Okounkov bodies}

   Functions on Okounkov bodies have been studied by Boucksom and Chen
   \cite{BouChe11} and Witt-Nystr\"om \cite{Nys09}. As their approaches
   differ, we present here briefly both of them, keeping in mind
   that we will be interested later on in continuous functions
   on Okounkov bodies. As Proposition~\ref{notcontinuous} shows,
   this is a quite delicate issue.

   We fix for duration of this section
   a  projective variety $X$ together with
   an admissible flag of subvarieties $\ybul:\, X=Y_0\supset\dots\supset Y_n$.

\subsection{Okounkov functions as concave envelopes}
   We begin with describing Witt-Nystr\"om's construction, in a slightly different way from \cite{Nys09}. 
   We recall first an auxiliary notion, see \cite[Section 7]{Rock}.
\begin{definition}[Closed concave envelope]\label{def:closed concave envelope}
   Let $\Delta\subset\R^n$ be a compact, convex set, and let $f:\Delta\to\R$
   be a bounded real valued function on $\Delta$. The \emph{closed concave envelope} $f^c$
of $f$ on $\Delta$ is defined by
   \[f^c(x)={\rm inf}\{g(x)| g\geq f, g\mbox{ concave and upper semi-continuous}\}.\]
   \end{definition}
   The closed
 concave envelope of a bounded function $f$ can be constructed as follows.
Let $H$ be the hypograph of $f$ in $\Delta\times \R$, let $H^c$ be the closed
convex hull of $H$ and define $f^c$ to be the unique function on $\Delta$ having $H^c$ as its hypograph, cf \cite{Rock}.

\begin{remark}\label{rmk:usc}
The function $f^c$ is concave and upper semi-continuous (since its hypograph is closed). From its concavity it
follows that $f^c$ is continuous in the interior of $\Delta$. Being
concave and upper-semi-continuous, it is continuous along any line segment lying
in $\Delta$.
\end{remark}

   From now on we work with a linearly bounded filtration $\fbul$ on $\vbul$
   (typically defined by a geometric valuation $\nu$ on the function field $\C(X)$).

We define a function $\wtilde{\phifbul}$ at points $v\in\dybul(\vbul)$ which
are   normalized valuation vectors by
\begin{equation}\label{eq:2}
   \wtilde{\phifbul}(v):=
      \lim_{k\to\infty}\frac1k\sup\left\{t\in\R:\; \exists s\in\calf_tV_k:\; \nuy(s)=k\cdot v\right\}.
\end{equation}This  limit exists because the sequence
   $$a_k:=\sup\left\{t\in\R:\; \exists s\in\calf_tV_k:\; \nuy(s)=k\cdot v\right\}$$
   is superadditive, i.e. $a_k+a_l\leq a_{k+l}$ for all $k,l\geq 1$.
   Indeed, let $\eps>0$ be fixed. There exist sections
   $$s_1\in\calf_{a_k-\eps/2}V_k\;\mbox{ and }\; s_2\in\calf_{a_k-\eps/2}V_l$$
   such that $\nuy(s_1)=kv$ and $\nuy(s_2)=lv$, so
$(s_1s_2)\in\calf_{a_k+a_l-\eps}V_{k+l}$ by the multiplicativity of the filtration
   and $\nuy(s_1s_2)=(k+l)v$. The existence of the limit now follows from
Fekete's Lemma \cite{Fek23}.


   In points $x$ which are not valuation vectors (in particular in such points
that do not belong to $\Delta_{\ybul}(\vbul)$) we set $\wtilde{\phifbul}(x):=0$.
   Thus the mapping $\wtilde{\phifbul}$ is defined on the whole space $\R^n$.
Now we are in a position to define
   the Okounkov function. 
\begin{definition}[Okounkov function 1]\label{def:of1}
   Using the above notation, we set
   $$\phifbul(x):=\wtilde{\phifbul}^c(x)$$
   for all $x\in\dybul(\vbul)$.
   We call this concave function the \emph{Okounkov function}
   associated to $\fbul$.

   If $\fbul$ is the filtration associated to a geometric valuation  $\nu$ of the function field of $X$, then we will also use the notation $\phi_\nu$ for $\phifbul$.
\end{definition}
   We observe now that taking concave envelope leaves the values of the
   underlying function $\wphifbul$ in normalized valuation vectors untouched.
\begin{lemma}
   For an arbitrary normalized
   valuation vector $v$ there is the equality
   $$\phifbul(v)=\wtilde{\phifbul}(v).$$
\end{lemma}
\proof
   It suffices to show  that the function $\wphifbul$ is
   ''concave'' on the normalized valuation vectors.
   To this end, it suffices to show
   \begin{equation}\label{eq:concave pre function}
      \frac12\wphifbul(v)+\frac12\wphifbul(u)\leq \wphifbul\left(\frac12u+\frac12v\right)
   \end{equation}
   for all normalized valuation vectors $u$ and $v$. Note that it follows
   from the proof of Lemma \ref{lem:density of valuation} that $\frac12(u+v)$
   is again a normalized valuation vector.
   \\
   Let $\eps>0$ be fixed.
   It follows from the discussion right after \eqnref{eq:2}
   that the limit in \eqnref{eq:2} is actually a supremum.
   Hence there exist numbers $k,l\in\N$ and $t_1,t_2\in\R$, as well as sections
   $s_1\in\calf_{t_1}V_k$, $s_2\in\calf_{t_2}V_l$ such that
   $\nuy(s_1)=ku$, $\nuy(s_2)=kv$ and
   $$\frac{t_1}{k}>\wphifbul(u)-\eps\; \mbox{ and }\; \frac{t_2}{l}>\wphifbul(v)-\eps.$$
   Then for $s=s_1^l\cdot s_2^k$ we have
   $$s\in V_{2kl}\; \mbox{ and }\; \nuy(s)=2lk\left(\frac12u+\frac12v\right).$$
   Moreover $s\in\calf_{lt_1+kt_2}V_{2kl}$ by the multiplicity of $\calfbul$.
   Hence
   $$\wphifbul\left(\frac12u+\frac12v\right)\geq\frac{lt_1+kt_2}{2lk}>
     \frac12\wphifbul(u)+\frac12\wphifbul(v)-\eps$$
   which implies \eqnref{eq:concave pre function}.
\endproof

\begin{remark}
In \cite{Nys09}, Witt-Nystr\"om  actually uses the following version of the above construction. For $(v,k)$ in the graded semigroup $\Gamma_{Y_\bullet}(V_\bullet)$ he sets
$$
f(v,k):=\sup\left\{t\in\R:\; \exists s\in\calf_tV_k:\; \nuy(s)=v\right\},
$$
which defines a super-additive function on $\Gamma_{Y_\bullet}(V_\bullet)$. Writing each $v\in\mathring{\Delta}_{Y_\bullet}(V_\bullet)$ as the limit of a sequence of the form $\e_k v_k$ with $\e_k\to 0_+$ and $(k,v_k)\in\Gamma_{Y_\bullet}(V_\bullet)$, he then proves that 
$$
\hat f(v):=\lim_{k\to\infty}\e_k f(v_k,k)
$$
exists in $\R$, only depends on $v$, and defines a concave function on $\mathring{\Delta}_{Y_\bullet}(V_\bullet)$. His arguments provide a several variable version of the classical 'Fekete lemma', and are presented in the Appendix for the convenience of the reader. 

When $v\in\mathring{\Delta}_{Y_\bullet}(V_\bullet)$ is a valuation vector, the definitions combined with the above lemma yield
$$
\hat f(v)=\wtilde{\phifbul}(v)=\phifbul(v), 
$$
and it follows by density that $\hat f$ and $\varphi_{\fbul}$ coincide on $\mathring{\Delta}_{Y_\bullet}(V_\bullet)$.
\end{remark}

\begin{remark} Keeping the notation from above, let $\fbul$ be the valuation obtained from a geometric valuation $\nu$, let $D$ be a big divisor. Then
 \[
  \inf_{\Delta_{\ybul}(D)} \phi_\nu \geq  \nu(\| D\|)\ ,
 \]
where $\nu(\| D\|)$ denotes the asymptotic value of $\nu$ on $D$ as defined in \cite[Section 2]{AIBL}.
\end{remark}

\begin{remark}\label{rmk:okfn usc}
 It is an obvious but important consequence of Remark~\ref{rmk:usc} that Okounkov functions are upper-semi-continuous.
\end{remark}

\begin{proofof}{Theorem~\ref{thm:cont}, Part (i)}
According to \cite[Proposition 3]{Howe}, all non-negative concave upper-semi-continuous functions are continuous on locally polyhedral subsets of $\R^n$. In
particular, if $\Delta_{\ybul}(\vbul)$ is a polytope (no matter whether it is rational or not), then all Okounkov non-negative Okounkov functions are
automatically continuous on the whole of $\Delta_{\ybul}(\vbul)$.

This latter statement includes in particular all Okounkov functions coming from geometric valuations.
\end{proofof}

\subsection{Okounkov functions via graded linear series}

   Here we recall the construction by Boucksom and Chen.
   Let $\fbul$ be a multiplicative
   filtration on the graded linear series $\vbul$. Then, for any
$t\in\R$, we can define  a new graded linear series $\vbul^{(t)}$ via
   \begin{equation}\label{eq:veronese filtr}
      V^{(t)}_k:=\calf_{tk}V_k
   \end{equation}
   for all $k$.
The Okounkov bodies $\dybul\left(\vbul^{(t)}\right)$ form a non-increasing
family of compact
   convex subsets of $\dybul(\vbul)$ and they have been used by Boucksom and Chen \cite[Definition 1.8]{BouChe11}
   in order to define functions
   on Okounkov bodies.

\begin{definition}[Okounkov function 2]\label{def:of2}
   With notation as above,  put
   $$\psifbul(x) \deq \sup\left\{t\in\R:\; x\in \dybul\left(\vbul^{(t)}\right)\right\}.$$
   for all $x\in\dybul(\vbul)$
   and call this function also the \emph{Okounkov function}
   associated to $\fbul$.
\end{definition}
The following lemma states that these two definitions are equivalent.
\begin{lemma}
   The definitions \ref{def:of1} and \ref{def:of2} are equivalent on $\dybul(\vbul)$, i.e.
   $$\phifbul(x)=\psifbul(x)$$
   for all $x\in\dybul(\vbul)$.
\end{lemma}
\proof
In what follows, we will denote the closed convex closure of a subset of $S\subset \R^n$
by $\clconv(S)$. Consider the set
\[ H_1=\{(x,y)| x\mbox{ a normalised valuation vector }, \exists v \mbox{ s.t. }
\nu(v)= x, \mbox{val}(v)\geq y\} \dsubseteq \Delta\times\R\ .
\]
Note that by definition
\[
\Delta_t \deq  \mbox{ closed convex hull }(H_1\cap \{\Delta \times t\})\ .
\]
In particular, if we consider the set $H_2\subset \Delta\times \R$ defined by
\[
H_2\cap \{\Delta\times t\} \equ  \mbox{ closed convex hull}(H_1\cap \{\Delta\times t\})
\]
then we have that
\[
\psifbul(x) \equ \sup\left\{ t\mid  (x,t)\in H_2\right\}\ .
\]
Observe  that $H_1\subset H_2\subset {\rm clconv}(H_1)$.

Let $H_3$ be the hypograph of $\psifbul$: we then have that
$H_2\subset H_3\subset \cl(H_2)$. Moreover, $H_3$, as the hypograph of an
upper semi-continuous concave function is automatically closed and convex,
so $H_3= \cl(H_2)$, and this closure is a convex set. In particular,
we have that $H_3=\cl(H_2)= \clconv(H_1)$.
Let $H$ be the hypograph of $\tilde \phifbul$, so that ${\rm clconv}(H)$ is the hypograph of
$\phifbul$. It is immediate from the definition
that
\[ H_1\subset H \subset {\rm cl}(H_1)\]
and hence ${\rm clconv}(H)={\rm clconv}(H_1)=H_3$. The hypograph
$\phifbul$ is therefore equal to $H_3$, which is the hypograph of
$\psifbul$. These two functions are therefore equal.
\endproof

\subsection{An example of a non-continuous Okounkov function}

Concave envelopes are in general only upper semicontinuous on the boundary. In the absence of good geometric properties of $\Delta_{\ybul}(\vbul)$ (cf.
Theorem~\ref{thm:cont} (i)),  there is no guarantee that Okounkov functions defined on $\Delta_{\ybul}(\vbul)$ will be continuous. Here we show by example that
such a situation indeed can occur.

First, the following Proposition gives a sufficient condition for non-continuous
   behavior of an Okounkov function. After its proof we present
   an example where the circumstances described  do happen.

\begin{proposition}[A non-continuity criterion]\label{notcontinuous}
    Let $X$ be a variety, $\ybul:\; X=Y_0\supset Y_1\supset\ldots\supset Y_n$ a flag on $X$
    and $D$ a divisor on $X$.
    Let $\dybul(D)$ be the Okounkov body of $D$ with respect to this flag and let
    $p$ be a point in the boundary of $\Delta(D)$ such that $p= \nu(s)$, where $s$
    is a section in $H^0(X,D)$ defining a reduced irreducible divisor $Y$ and $\nuy$
    is the multivaluation associated to the flag $\ybul$.
    Let $v$ be the
    valuation associated to $Y$, i.e. $v=\ord_Y$. \\
If $\dybul(D)$ is not locally a cone around $p$, then the Okounkov
function $\varphi_v$ associated to the valuation $v$ is
not continuous at the point $p$.
\end{proposition}
\begin{proof}
Let us consider the Okounkov bodies $\Delta_t(D)$ associated to the filtration
given by the valuation $v$.
We have that $\Delta_t(D)= t\nuy(s)+ (1-t) \dybul(D)$ for $t\in[0,1]$ and
$\Delta_t(D)=\emptyset$ if $t>1$.
In other words, if $t\in[0,1]$ then
$\Delta_t(D)$ is produced from $\dybul(D)$ by performing on
$\dybul(D)$
a homothety of ratio $(1-t)$ centered at the point $p=\nuy(s)$. \\ \\
In particular,
$p\in \Delta_t(D)$ for all $t\in[0,1]$ and hence
\[\varphi_v(p)=\sup \{ t:\; p\in
\Delta_t(D)\}=1.\]
Since $\dybul(D)$ is not locally a cone around $p$ we can find a sequence of
points $p_i$ contained in the boundary $\partial\dybul(D)$ such that
\begin{enumerate}
\item $\lim_{i\rightarrow \infty} p_i=p$;
\item For all integers $i$ and $t>0$ we have that
$p+(1+t)(p_i-p)\not\in \dybul(D)$.
\end{enumerate}
In other words, the line passing through $p$ and $p_i$ leaves the Okounkov body
$\dybul(D)$ exactly at the point $p_i$. In particular, it follows that
$p_i\not\in \Delta_t(D)$ for any $t>0$ so that $\phi_v(p_i)=0$. It follows that
$\phi_v$ is not continuous at the point $p$. This completes the proof of
Proposition \ref{notcontinuous}.
\end{proof}

We will now produce a threefold $X$ along  with an admissible flag
$X=Y_0\supset Y_1\supset Y_2\supset Y_3$, a divisor $D$ on $X$ and a section $s$ of $D$, such that
$\nu(s)$ lies in the round part of the boundary of $\Delta(D)$.

Our example comes from \cite{KLM}, which is in turn  heavily based on earlier work of  Cutkosky \cite{C}.
The first part of the discussion is taken from \cite{KLM} almost verbatim.

In \cite{C}, Cutkosky constructs a quartic
surface $S\subseteq\mathbb{P}^3$ whose N\'eron-Severi space $N^1(S)$ is isomorphic to
$\mathbb{R}^3$ with the lattice $\mathbb{Z}^3$ and the intersection form $q
(x,y,z)= 4x^2-4y^2-4z^2$. He shows that

\begin{enumerate}
\item The divisor class $(1,0,0)$  on $S$ corresponds to a very ample divisor
class $[L]$
and the projective embedding  corresponding to $L$ realizes  $S$ as a quartic
surface in $\mathbb{P}^3$.
\item The nef and effective cones of $S$ coincide, and are  given by the
conditions
\[
v^2\geq 0\ ,\  ([L]\cdot v)> 0\ .
\]
\end{enumerate}

Now, take the nef class $\alpha \deq (1,1,0)\in N^1(S)$, and let $C$ be a
curve with class $[C]=\alpha$. We note that since the effective cone of $S$ has no
polyhedral part, any curve $C$ on $S$ such that $C^2=0$ and $\frac1k[C]$ is not
integral for any $k>1$, is automatically irreducible. In this case, all
members of the linear series of $C$ are irreducible.

Since  $C^2=0$,  Riemann-Roch implies that $\chi(C)=2$ hence
$h^0(C)+h^2(C)= h^0(C)+ h^0(-C)\geq 2$.
As  $(L\cdot (-C))=-4$, we know that $h^0(-C)=0$ and it follows that
$h^0(C)\geq 2$. There is therefore a pencil of curves  on $S$ with the class $\alpha$,
no two different elements of  this pencil  meet because $\alpha^2=0$
and all members of the pencil are irreducible.
This pencil is hence  base point free and its general element is
smooth by Bertini theorem. A general
element $C\subset \P^3$ of this pencil is then a smooth elliptic curve of
degree 4.

   Let $X$ to be the blow-up of $\mathbb{P}^3$ along the curve $C$. We denote
by $Y_1\subset X$ the proper transform of $S$ in $X$.
We note that $Y_1$ is
isomorphic under projection to $S$.

We now choose a sufficiently positive ample divisor $D$ on $X$, such that
$D|_{Y_1}$ and $Y_1|_{Y_1}$ are independent in the Picard group of $Y_1$.
   Moreover we can assume that
   all of the following divisors
   are ample:
\begin{equation}\label{eq:positive divisors}
   D\, , \, \mbox{ }D-Y_1\, , \, \mbox{ }D-Y_1-K_X\, , \, \mbox{ } D-2Y_1-K_X\ .
\end{equation}
   Furthermore, we choose a curve $C'$ on $Y_1=S$ such that $[C']$ is
a primitive integral member of the boundary of
${\rm Eff}(Y_1)$. (The class $[C']$ is effective by the Riemann-Roch argument
given above). Moreover, we assume
that $C'$ is not contained in the image of the restriction map from ${\rm Pic}(X)$ to
${\rm Pic}(Y_1)$. We note that this implies that
\begin{equation}\label{eq:not in plane}
   [D|_{Y_1}], [Y_1|_{Y_1}] \;\mbox{ and }\; [C'] \;\mbox{ are independent in }\; NS(Y_1).
\end{equation}
    Finally, we pick $Y_2$ to be
a smooth curve contained in the class $D|_{Y_1}-C'$ and pick $Y_3$ to be a
general point on $Y_2$.

\begin{proofof}{Theorem~\ref{thm:cont}, (ii)}
 With $X$, $Y_1,Y_2,Y_3$ and $D$ as above, we now show that there is a reduced and irreducible
divisor $Z$ on $X$, linearly equivalent to $D$, such that close to the point
$\nu(Z)\in\Delta(D)$ the set $\Delta(D)$ is not locally a cone. Here $\nu$
   denotes the $3$--valuation determined by the flag $\ybul$.

   Since $D$ and  $D-Y_1-K_X$ are both ample by \eqnref{eq:positive divisors},
   the restriction map on global sections $H^0(D)\rightarrow H^0(D|_{Y_1})$ is
surjective, and indeed so is $H^0(kD)\rightarrow H^0(kD|_{Y_1})$
for any $k$.

We can therefore choose a section of $D$ determining a divisor $Z$
not vanishing along $Y_1$ and such that $Z|_{Y_1}= Y_2\cup C'$. By generality of $Y_3$
we then have
\[
\nu(Z) \equ (0,1,0)\ .
\]
Let us show now that the divisor $Z$ is reduced and irreducible. If not then
we can write $Z$ as a sum of non-zero effective divisors
\[
Z \equ  Z'+Z''\ .
\]
   Then $Z'|_{Y_1}$ and $Z''|_{Y_1}$ are non-zero effective divisors and
$(Z'+Z'')|_{Y_1}= C'+Y_2$, where $C'$ and $Y_2$ are both irreducible.
Without loss of generality $Z'|_{Y_1}=C'$, but this
contradicts our assumption that $C'$ is not a restriction of a divisor on
$X$.

Let us now show that $\Delta(D)$ is not locally a cone at $(0,1,0)$.
We consider
\[
\Delta'(D) \equ  \{(a,b,c)| 0\leq a\leq 1, (a,b,c)\in \Delta(D)\}
\]
i.e. we consider a part of $\Delta(D)$ with $a$ sufficiently small.
   From \eqnref{eq:positive divisors} it follows that for any $k$ and any $a\in[0,1]$ such that $ka
\in \mathbb{N}$ the mapping
\[
H^0(k(D-aY_1))\rightarrow H^0(k(D-aY_1)|_{Y_1})
\]
is surjective. It follows that for any $a\in[0,1]$
\[\Delta(D)\cap \{(a,-,-)\}= \{ (a, b,c)| (b,c)\in \Delta((D-aY_1)|_{Y_1})\}.\]
In other words, the slice of the Okounkov body $\Delta(D)$ with the plane
$(a,-,-)$ is just the Okounkov body of $(D-aY_1)|_{Y_1}$ on $Y_1$. \\ \\
As $Y_1$ is a surface with no negative curves, the description of its Okounkov
bodies given in \cite[Theorem 6.4]{LazMus09} is then very simply
\[ \Delta(D-aY_1)|_{Y_1}=\{ (b,c)| (D-aY_1)|_{Y_1}-bY_2 \mbox{ effective },
0\leq c\leq (D-aY_1-bY_2)\cdot Y_2\}\]
or in other words
\[
 \Delta'(D-aY_1)=\{ (a,b,c)| 0\leq a\leq 1,f_1\geq 0,f_2>0,0\leq c\leq f_3\}\ ,
\]
where
\begin{eqnarray*}
 f_1 & = & (D|_{Y_1}-aY_1|_{Y_1}-bY_2)^2 \ ,\\
f_2 & = &  (D|_{Y_1}-aY_1|_{Y_1}-bY_2) \cdot L\ ,\\
f_3 & = & (D|_{Y_1}-aY_1|_{Y_1}-bY_2) \cdot Y_2\ .
\end{eqnarray*}

For simplicity, let us now consider the slice
\[\Delta''(D)=\{ (a,b)| (a,b,0)\in \Delta_\epsilon(D)\}\]
obtained by intersecting $\Delta'(D)$ and the plane $c=0$. Alternatively,
we can write
\[  \Delta''(D)=\{ (a,b)|0\leq a \leq 1,(D|_{Y_1}-aY_1|_{Y_1}-bY_2)^2\geq
0 \\ \mbox{ and }(D|_{Y_1}-aY_1|_{Y_1}-bY_2) \cdot L>0 \}.\]
It will be enough to show
that $\Delta''(D)$
is not locally a cone around the point $(0,1)$. Recall that
any cone in $\mathbb{R}^2$ is either the whole of $\mathbb{R}^2$ or is bounded
by two straight half-lines. $(0,1)$ is not an interior point of
$\Delta''(D)$ so the first possibility is excluded. \\ \\
The set $\Delta''(D)$ is bounded by the following curves:
\begin{enumerate}
\item the $x$-axis,
\item the $y$-axis,
\item the line $x=\epsilon$,
\item the branch of the conic section defined by the equation
\[(D|_{Y_1}aY_1|_{Y_1}-bY_2)^2=0.\]
passing through $(0,1)$.
\end{enumerate}
The point $(0,1)$ lies at the intersection of the $y$-axis and the conic
section defined by the equation $(D|_{Y_1}-aY_1|_{Y_1}-bY_2)^2=0$.\\ \\
To establish that $\Delta''(D)$ is not locally a cone around $b$ it will be
enough to show that the conic section given by the equation $(D|_{Y_1}-
aY_1|_{Y_1}-bY_2)^2=0$
does not contain a straight line. This conic section is the intersection in
${\rm Pic}(Y_1)$ of the
nef cone $x^2=y^2+z^2$ with the plane passing through the points $D|_{Y_1}$, $(D-Y_1)|_{Y_1}$ and
$D|_{Y_1}-Y_2$. By \eqnref{eq:not in plane} this plane does not pass through $0$ so the resulting
conic section is not a union of straight lines.
This completes the proof of Theorem~\ref{thm:cont}.
\end{proofof}

\subsection{Examples}

We devote this section to several examples where functions associated to various geometric valuations are determined explicitly.
First we deal with the one-dimensional case, where Okounkov functions associated to complete linear series can be computed in general.

\begin{example}[Okounkov function of a valuation on a curve]\rm
   Let $C$ be a smooth curve, $\vbul$ a big graded linear system
   associated to a line bundle $L$ of positive degree,  and let
   $$\ybul:\; C\supset\left\{p\right\}$$
   be a fixed flag.\\
   a) Consider the filtration $\fbul=\ord_p$ on $\vbul$ defined by the order
   of vanishing at the point $p$ in the flag.

   Let $x\in\Delta_{\ybul}(\vbul)$ be arbitrary, and write it as  a limit
   of normalized valuation vectors $x=\lim_{k\to\infty}\frac{\alpha_k}{k}$. Then
\begin{eqnarray*}
 \phi_{\ord_p}(x) & = & \lim_{k\to \infty}\frac1k\sup\left\{t\in\R:\;
   \exists s\in V_k:\; \ord_p(s)\geq t\mbox{ and } \ord_p(s)=k\alpha_k\right\} \\
& = & \lim_{k\to\infty}\frac1k\alpha_k \\
& = & x\ .
\end{eqnarray*}
   It turns out  that in this case the Okounkov function is the identity.\\
   b)  Now  consider the filtration $\fbul=\ord_q$ defined by the
   order of vanishing in a point $q$ not in the flag. In this case, we take $\vbul$ to be the complete graded linear series associated to the divisor $L$.
   At a point $x=\lim_{k\to\infty}\frac{\alpha_k}{k}$ as above, we have
\begin{eqnarray*}
 \phi_{\ord_q}(x) & = & \lim_{k\to \infty}\frac1k\sup\left\{t\in\R:\;
   \exists s\in V_k:\; \ord_q(s)\geq t\mbox{ and } \ord_p(s)=k\alpha_k\right\} \\
& = & \lim_{k\to\infty}\frac1k(k\deg(L)-\alpha_k) \\
& = & \deg (L) -x\ .
\end{eqnarray*}
\end{example}

Next, we move on to the surfaces, where calculations become very difficult very soon. This is not surprising, since invariants of Okounkov functions on
surfaces already carry  deep geometric information (see \cite{KMS}).

\begin{example}[Okounkov function of a valuation on the projective plane]\label{ex:ok f on p2}\rm
   Set $X_0=\P^2$, $D_0=\calo_{\P^2}(1)$, and let $P_0\in\ell\subset X_0$ be a flag
   as in Example \ref{ex:ok of p2 and blow up}.\\
   a). First, we handle the case  $\nu=\ord_{P_0}$. In the
   rational points $(a,b)\in\Delta(D_0)$ the Okounkov function $\phi^0$ is then
\begin{eqnarray*}
\phi^0(a,b) & = & \lim_{k\to\infty}\frac1k\sup\left\{t\in\R\mid\exists s\in|kD_0|:\;
   \ord_{\ell}(s)=ka,\; \ord_{P_0}(s_1)=kb,\right. \\
   && \left. \ord_{P_0}(s)\geq t\right\} \\
& = &  \lim_{k\to\infty}\frac1k k(a+b) \\
& = & (a+b)\ ,
\end{eqnarray*}
   where $s_1$ is defined as in \eqnref{eq:s1}. As the Okounkov body $\Delta(D_0)$ is a polytope, $\phi^0$ is continuous by Theorem~\ref{thm:cont}, hence
$\phi^0(a,b)=a+b$
   for all $(a,b)\in\Delta(D_0)$. We point out that using the definition of Boucksom and Chen, one can obtain the result without referring to the continuity of
$\phi^0$. \\
   b) Now we consider $\nu=\ord_{P_1}$ for a point $P_1$ not on the line $\ell$.
   For the rational points $(a,b)\in\Delta(D_0)$ we have
\begin{eqnarray*}
\phi^1(a,b) & = & \lim_{k\to\infty}\frac1k\sup\left\{t\in\R:\; \exists s\in|kD_0|:\;
   \ord_{\ell}(s)=ka,\; \ord_{P_0}(s_1)=kb,\right. \\
   && \left. \ord_{P_1}(s)\geq t\right\} \\
&  =  & \lim_{k\to\infty}\frac1k k(1-a) \\
& = & 1-a\ .
\end{eqnarray*}
Again, the same formula holds over the whole of $\Delta(D_0)$ by a continuity argument.
\end{example}

Note that the analogous calculations can be carried out on a projective space of arbitrary dimension.

\begin{example}[Okounkov function on a blow up of the projective plane]\label{ex:ok f on blow up of p2}
Keeping the notation of the Example~\ref{ex:ok f on p2}, let $f:X_1=\Bl_{P_1}X_0\to\P^2$ be the blow up of the projective plane
   in a point $P_1$ not contained in the flag line $\ell$ with exceptional divisor $E_1$.
   We work now
   with a $\Q$--divisor $D_{\lambda}=f^*(\calo_{\P^2}(1))-\lambda E_1$,
   for some fixed $\lambda\in [0,1]$. A direct computation using \cite[Theorem 6.2]{LazMus09} gives that the Okounkov body has the shape
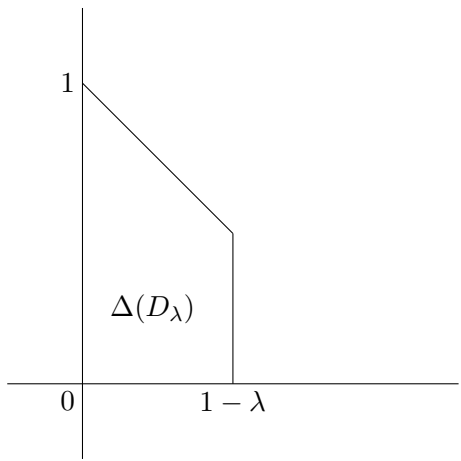
\begin{center}
\begin{pgfpicture}{0cm}{0cm}{6cm}{6cm}
   \pgfline{\pgfxy(0,1)}{\pgfxy(6,1)}
   \pgfline{\pgfxy(1,0)}{\pgfxy(1,6)}
   \pgfline{\pgfxy(1,5)}{\pgfxy(3,3)}
   \pgfline{\pgfxy(3,1)}{\pgfxy(3,3)}
   \pgfputat{\pgfxy(3,0.9)}{\pgfbox[center,top]{$1-\lambda$}}
   \pgfputat{\pgfxy(0.9,0.9)}{\pgfbox[right,top]{$0$}}
   \pgfputat{\pgfxy(0.9,5)}{\pgfbox[right,center]{$1$}}
   \pgfputat{\pgfxy(2.5,2)}{\pgfbox[right,center]{$\Delta(D_{\lambda})$}}
\end{pgfpicture}
\end{center}
   a). For the valuation $\nu=\ord_{P_0}$, we get exactly as above
   $$\phi^0(a,b)=a+b.$$\\
   b). For the valuation $\nu=\ord_{P_2}$, where $P_2$ is a point in $X_1$
   not on the exceptional divisor $E_1$ (hence $P_2$ can be considered also
   as a point on $\P^2$) and not on the line through $P_0$ and $P_1$. We have now
   for $(a,b)\in\Delta(D_{\lambda})$
   $$\phi^1(a,b)=\left\{\begin{array}{ccc}
      1-a & \mbox{ for } & a+b\leq 1-\lambda\\
      2-2a-b-\lambda & \mbox{ for } & 1-\lambda\leq a+b\leq 1
      \end{array}\right.$$
  This can be seen as follows. $\phi^1(a,b)$ is the maximal
   order of vanishing at $P_2$ among all $\Q$--sections vanishing
   \begin{itemize}
      \item[a)] along $\ell$ to order $a$;
      \item[b)] in $P_1$ to order $\lambda$;
      \item[c)] in $P_0$ to order $b$ after dividing by the
      equation of $\ell$ in power $a$ and after restricting to $\ell$.
   \end{itemize}
   Condition a) ''costs'' $aH$, so we are left with
   $(1-a)H-\lambda E_1$ to take care of conditions b) and c).
   If $b\leq 1-a-\lambda$, then we take a line through the
   points $P_2$ and $P_1$ with multiplicity $\lambda$
   and the line through $P_2$ and $P_0$ with multiplicity $1-a-\lambda$.
   Their union has multiplicity $\lambda+(1-a-\lambda)=1-a$ at $P_2$
   and satisfies b) and c). Moreover, there is no $\Q$-divisor
   equivalent to $(1-a)H-\lambda E_1$ with
   higher multiplicity at $P_2$, which follows easily from
   B\'ezout's theorem intersecting with both lines.

   The argument in the remaining case $b>1-a-\lambda$ is similar.
   We want to split the divisor so that it produces a high
   vanishing order towards condition c) first and then,
   after arriving to the threshold
   \begin{equation}\label{eq:threshold}
      b'=1-a'-\lambda',
   \end{equation}
   we take again the union of two lines as above. Thus,
   we start with the conic through $P_1$ and $P_2$
   tangent to $\ell$ at $P_0$. We take this conic with
   multiplicity $\alpha$ subject to condition that
   $$b-2\alpha=1-a-2\alpha-(\lambda-\alpha),$$
   which means that the divisor $(1-a-2\alpha)H-(\lambda-\alpha)E_1$
   satisfies \eqnref{eq:threshold} with $b'=b-2\alpha$,
   $a'=a+2\alpha$ and $\lambda'=\lambda-\alpha$.
   The constructed $\Q$-divisor, consisting of the conic and two lines
   has then multiplicity
   $$a+b+\lambda-1+(1-a-2(a+b+\lambda-1))=2-2a-b-\lambda.$$
   B\'ezout's theorem shows then that there is no divisor of higher
   multiplicity.

\end{example}

\subsection{Invariants of Okounkov functions}

We treat various properties of Okounkov functions.

   Given a linearly bounded filtration $\fbul$ on a graded linear series $\vbul$,
   we can restrict it to $\fmbul$ on the Veronese subseries
   $$\vmbul:=\bigoplus_{k=1}^\infty V_{mk}$$
   for $m\geq 1$. The index $m$ in $\fmbul$ helps us to keep track to which
   graded linear series the valuation is applied in the given moment.
   The corresponding Okounkov bodies scale well
   by \cite[Proposition 4.1]{LazMus09}
   $$\Delta_{\ybul}(\vmbul)=m\Delta_{\ybul}(\vbul),$$
   so that it makes sense to compare the corresponding Okounkov functions.
   It turns out that they scale as well.
\begin{theorem}[Veronese homogenity of Okounkov functions]
   Let $X$ be an irreducible projective variety and let $\fbul$ be a linearly
   bounded valuation on the graded linear series $\vbul$. Then
   \begin{equation}\label{eq:4}
      \phifmbul(mx)=m\cdot\phifbul(x)
   \end{equation}
   for all $x\in\Delta_{\ybul}(\vbul)$.
\end{theorem}
\proof
   To begin with let $x\in\Delta_{\ybul}(\vbul)$ be a normalized
   valuation vector. Then
   $$\phifmbul(mx)=\sup\left\{t\in\R:\;
     \exists s\in \calf_{mt}V_k(mL):\; \nu_{\ybul}(s)=m x\right\}$$
   $$=\sup\left\{t\in\R:\;
     \exists s\in \calf_{mt}V_{mk}(L):\; \nu_{\ybul}(s)=m x\right\}$$
   $$=m\sup\left\{t\in\R:\;
     \exists s\in \calf_tV_{k}(L):\; \nu_{\ybul}(s)=x \right\}
     =m\,\phifbul(x).$$
   The equality of both functions follows then from the density
   statement \ref{lem:density of valuation} (1.) and the fact
   that the closed concave envelope is unique.
\endproof
   Using the above result we show that working with an appropriate flag,
   the Okounkov function can be recovered out of its values on the boundary
   of $\Delta_{\ybul}(\vbul)$. More precisely, we establish the following fact.
\begin{theorem}[Reading off Okounkov functions from the  boundary]\label{thm:function from boundary}
   Assume that $\vbul$ is a graded linear series associated to the line bundle
   $L$ such that there is an irreducible divisor $Y_1\in|L|$. We take a flag $\ybul$
   whose divisorial part is $Y_1$. Let $\fbul$ be a filtration on $\vbul$
   defined by a geometric valuation $\nu$. Then for $x=(x_1,\dots,x_n)\in\Delta_{\ybul}(\vbul)$
   we have
\begin{equation}\label{eq:5}
   \phifbul(x_1,\dots,x_n)=(1-x_1)\phifbul\left(0,\frac{x_2}{1-x_1},\dots,\frac{x_n}{1-x_n}\right)
      + x_1\cdot\nu(Y_1).
\end{equation}
\end{theorem}
\proof
   It suffices to establish \eqnref{eq:5} in case $x=v$ is a normalized valuation vector.
   For $m$ large enough and divisible all coordinates $mx_1$,...,$mx_n$ are integers and we have
   by \eqnref{eq:4}
\begin{equation}\label{eq:6}
   \phifbul(x)=\frac1m\phifmbul(mx)=
      \frac1m\lim_{k\to\infty}\frac1k\sup\left\{t:\;\exists s\in V_{mk}:\;
      \nu(s)\geq t \mbox{ and } \nu_{\ybul}(s)=mkx\right\}.
\end{equation}
   A section $s$ with $\nu_{\ybul}(s)=mkx$ can be written as
   $s=s'\cdot s_1^{mkx_1}$, where $s_1\in H^0(L)$
   is the section defining $Y_1$ and we have
   $$\nu(s)=\nu(s')+mkx_1\cdot\nu(s_1),$$
   since $\nu$ is a geometric valuation. For the Okounkov valuation
   $\nu_{\ybul}$ we have
   $$\nu_{\ybul}(s_1^{mkx_1})=(mkx_1,0,\dots,0)\;\mbox{ and }\;
     \nu_{\ybul}(s')=(0,mkx_2,\dots,mkx_n)=:x'.$$
   Note that $s'$ is a section in $V_{(1-x_1)mk}$. Thus, continuing \eqnref{eq:6}
   we establish
\begin{eqnarray}\label{eq:7}
   \phifbul(x) & = & \frac1m\lim_{k\to\infty}\frac1k\left[
   mkx_1\nu(s_1)+\sup\left\{t:\;\exists s'\in V_{(1-x_1)mk}:\;
      \nu(s')\geq t \text{ and } \right. \right. \nonumber \\
      && \left. \left.  \nu_{\ybul}(s')=mkx'\right\}\right]\ .
\end{eqnarray}
   With $x'':=\frac{1}{1-x_1}\cdot x'$ we have
   $$mkx'=(1-x_1)mk\cdot x''$$
   and thus continuing \eqnref{eq:7} we have
   \begin{eqnarray*}
   \phifbul(x) & = & x_1\nu(s_1) +(1-x_1)\frac{1}{m(1-x_1)} \times \\
   && \times  \lim_{k\to\infty}\frac1k
   \sup\left\{t:\;\exists s'\in V_{(1-x_1)mk}:\;
      \nu(s')\geq t \right. \\
      && \left. \text{ and }\nu_{\ybul}(s')=(1-x_1)mkx''\right\} \\
      & = & x_1\nu(s_1)+(1-x_1)\frac1{m(1-x_1)}\varphi_{\calf_{m(1-x_1)\bullet}}(m(1-x_1)x'') \\
      & = & x_1\nu(s_1)+(1-x_1)\phifbul(x'')\ .
   \end{eqnarray*}
\endproof

\begin{remark}
Repeated applications of Theorem~\ref{thm:function from boundary} reduce the computation of $\phi_{\fbul}$ to the situation where we consider only those global sections of $L$ that vanish
 at the point $Y_n$. If the restriction map
\[
 \HH{0}{X}{\OO_X(mL)} \lra \HH{0}{Y_{n-2}}{\OO_{Y_{n-1}}(mL)}
\]
is surjective for $m\gg 0$, then this amounts to a calculation on the curve $Y_{n-1}$.  Consequently, the computation of $\phi_{\fbul}$ for very ample divisors can be essentially reduced to
the curve case.
\end{remark}

At last we check that the functions $\phi_\nu$ are continuous when considered as functions on the interior of the global Okounkov body of $X$.

\begin{proposition}\label{prop: global cont}
Let $X$ be an irreducible projective variety, $\nu$ a geometric valuation of $\C(X)$, $\phi_\nu:\Delta_{\ybul}(X)\to \R_{\geq 0}$ the associated Okounkov function. Then $\phi_\nu$ is
continuous on the open subset
\[
 U \deq \bigcup_{\alpha\in\Bbig(X)} \Delta_{\ybul}^{\circ}(\alpha) \dsubseteq N^1(X)_\R \ .
\]

\end{proposition}
\begin{proof}
Let $D_1,\dots,D_\rho$ be integral divisors on $X$ whose numerical equivalence classes form a $\Z$-basis of $N^1(X)_\R$; assume in addition that every effective divisor on $X$ is a non-negative
integral linear combination of the $D_i$'s up to numerical equivalence. This can be arranged by \cite[p.30.]{LazMus09}. For an element $\overline{m}\in\N^\rho$, we set as usual
\[
 \overline{m}\cdot \overline{D} \deq \sum_{i=1}^{\rho}m_iD_i\ .
\]
The multigraded semigroup of $X$ (with respect to the choices of the $D_i$'s and an admissible flag) is
\[
 \Gamma_{\ybul}(X) \equ \st{(\overline{m},\nu_{\ybul}(s))\,|\, 0\neq s\in\HH{0}{X}{\OO_X(\overline{m}\cdot\overline{D}}} \dsubseteq \N^{n+\rho}\ .
\]
The global Okounkov body of $X$ is then the closure of the convex hull  of the set of normalized multigraded valuation vectors
\[
\bigcup_{\overline{q}\in\Q_{\geq 0}^\rho}\bigcup_{k\in\N\, ,\, k\overline{q}\in\N^\rho} \st{\left(\overline{q},\frac{1}{k}\nu_{\ybul}(s)\right)\,:\; 0\neq s\in\HH{0}{X}{\OO_X(k\cdot\overline{q}\cdot\overline{D})}} \dsubseteq \R^{n+\rho}\ .
\]
If $(\overline{q},\alpha)$ is such a vector, then we define
\[
\wtilde{\phi}_\nu(\overline{q},\alpha) \deq
\lim_{k\to\infty}\frac1k\sup\left\{t\in\R:\; \exists s\in\calf_t\HH{0}{X}{\OO_X({k\cdot\overline{q}}\cdot\overline{D})}:\; \nuy(s)=k\cdot \alpha\right\}.
\]
For all other points of $\R^{n+\rho}$ we set $\wtilde{\phi}_\nu$ to be equal to zero. The concave transform of $\wtilde{\phi}_\nu$ is then a continuous function on $\Delta_{\ybul}(X)$,
which agrees over all  classes $\xi\in\N^1(X)_\Q$ with $\phi_\nu$ defined on the Okounkov body $\Delta_{\ybul}(\xi)$. This proves the claim.
\end{proof}

\section{Integrals of Okounkov functions}

In this section we point out a new way of constructing invariants of numerical equivalence classes of Cartier divisors via integrating functions on Okounkov
bodies. Let $X$ be an irreducible projective variety of dimension $n$ as so far, $Y_\bullet$ an admissible flag.

\begin{definition}
Let $\vbul$ be a big graded linear series, $\nu$ a geometric valuation of $\C(X)$.  We set
\[
 I(\vbul;\nu) \deq \int_{\Delta_{\ybul}(\vbul)} \phi_\nu\ .
\]
As usual, we write $I(D;\nu)$, whenever $\vbul$ is the complete graded linear series associated to a Cartier divisor $D$ on $X$.
\end{definition}

\begin{remark}
 The function $\phi_\nu$ is a bounded upper-semicontinuous concave function on the compact subset $\Delta_{\ybul}(\vbul)$, therefore it is Lebesgue
integrable. Being non-negative as well, its integral is non-negative, and so $0\leq I(D;\nu) <\infty$.
\end{remark}

It follows from results of \cite{BouChe11} that  $I(D;\nu)$ is in fact  independent of the flag $Y_\bullet$ as the notation suggests.

\begin{proposition}\label{prop:integral}
 With notation as above,
\[
 I(D;\nu) \deq \vol{\R^n}{\hat{\Delta}(V_\bullet, F_\nu)} \equ \int_{t=0}^{+\infty} \vol{\R^n}{\Delta(V^{(t)})} dt \equ \lim_{k\to +\infty} \frac{\mass
(V_k,F_\nu)}{k^{n+1}}\ .
\]
\end{proposition}
\begin{proof}
 This is the content of \cite[Corollary 1.11]{BouChe11}. Observe that the right-hand side expression is by its definition independent of the flag $Y_\bullet$.
\end{proof}

\begin{example}
   Let $f:X_1\to\P^2$ be the blow up of $\P^2$ in a point $P_1$ with the exceptional
   divisor $E_1$, as in Example \ref{ex:ok f on blow up of p2}.
   A divisor
   $$D=\alpha f^*(\calo_{\P^2}(1))-\beta E_1 \mbox{ is big on } X_1
   \mbox{ iff } \alpha>0 \mbox{ and } \beta<\alpha.$$
   For $\beta<0$ we have
   $$\Delta(\alpha f^*(\calo_{\P^2}(1))-\beta E_1)=\Delta(\alpha f^*(\calo_{\P^2}(1))),$$
   so it is enough to consider $0\leq \beta<\alpha$. Furthermore, we can divide by $\alpha$, see
   Remark \ref{rem:homogenity of I(D nu)}. Then with $\lambda=\frac{\beta}{\alpha}$,
   it follows from Example \ref{ex:ok f on blow up of p2}
   that
   $$I(f^*(\calo_{\P^2}(1))-\lambda E_1,\ord_{P})=\frac13-\frac12\lambda^2+\frac16\lambda^3$$
   for $P$ as in case a). or b). of that example.
\end{example}

\begin{remark}
Witt-Nystr\"om points  out  in \cite[Section 6]{Nys10} that test configurations of a pair $(X,L)$ defined by Donaldson \cite{Don1}
(where $X$ is a projective variety, $L$ an ample Cartier divisor on $X$) give rise to filtrations  of the section ring $R(X,L)$,
 and therefore to functions on $\Delta_{\ybul}(L)$ for some flag $\ybul$.

 Witt-Nystr\"om also observes (see \cite[Corollary 6.6]{Nys10}) that the integral of such a function (properly normalized) recovers the Futaki
 invariant $F_0$ of the test configuration one starts out with.
\end{remark}

The next statement fits in well with the philosophy that asymptotic invariants of line bundles tend to respect  numerical equivalence.

\begin{proposition}[Numerical invariance of Okounkov functions]\label{prop:num equiv}
Let $v$ be a discrete valuation, $D$ a big integral Cartier divisor on $X$. Then the function
\[
 \phi_v : \Delta_{Y_\bullet}(D) \lra \R
\]
depends only on the numerical equivalence class of $D$.
\end{proposition}
\begin{proof}
Fix an arbitrary numerically trivial divisor $P$ on $X$. First of all, as observed in \cite[Proposition 4.1]{LazMus09}, Okounkov bodies are invariant with
respect to numerical equivalence of divisors,
\[
 \Delta_{\ybul}(D) \equ \Delta_{\ybul}(D+P)\ ,
\]
whence the respective  domains of the functions $\phi_{D,\nu}$ and $\phi_{D+P,\nu}$ agree.

Next, following the train of thought of the proof of \cite[Proposition 4.1 (i)]{LazMus09}, we show that
\[
 \Delta(|\bullet (D+P)|^{(t)}) \equ \Delta(|\bullet D|^{(t)})
\]
holds for every $t\in\R$.

We recall  that $\Delta(|\bullet D|^{(t)})$ is the Newton--Okounkov body attached to the
graded linear series
\[
 A_k \deq \left\{s\in\HH{0}{X}{\OO_X(kD)}\,|\, v(s)\geq tk\right\}\ ,
\]
while $\Delta(|\bullet (D+P)|^{(t)})$ is the convex body associated to the graded linear series
\[
B_k \deq \left\{s'\in\HH{0}{X}{\OO_X(k(D+P))}\,|\, v(s')\geq tk\right\}\ .
\]
It follows from a Castelnuovo--Mumford regularity argument (see \cite[Lemma 2.2.42]{PAG}) that there exists a divisor $B$ on $X$ such that $B+lP$ is very ample
for all $l\in\Z$.
Let $a\gg 0$ be such that $|aD-B|\neq \emptyset$, and let $s\in\HH{0}{X}{\OO_X(aD-B)}$ be the section corresponding to an effective divisor. We write
\[
 (k+a)(D+P) \sim kD+(aD-B)+(B+(k+a)P)\ .
\]
If we represent $B+(k+a)P$ by a section not going through the elements of $Y_\bullet$, then we obtain
\[
  A_{k} \cdot s \subseteq B_{k}\ ,
\]
hence
\[
 \Gamma(A_k) + \nu(s) \subseteq \Gamma(B_{k})\ .
\]
By taking limits we obtain
\[
 \Delta_{\ybul}(|\bullet D|^{(t)}) \equ \Delta_{\ybul}(A_\bullet)\dsubseteq \Delta_{\ybul}(B_\bullet) \equ \Delta_{\ybul}(|\bullet (D+P)|^{(t)})\ .
\]
Replacing $D$ by $D+p$ and  $P$ by $-P$ in the above argument yields the reverse inclusion.
\end{proof}

\begin{proposition}\label{prop:integral cont}
Let $X$ be an irreducible projective variety, $\nu$ a geometric valuation of its function field. Then both
\[
 I(\cdot\, ,\nu):\Bbig(X)\lra \R_{\geq 0}\ \ \text{and}\ \ \frac{1}{\vol{X}{\cdot }}\cdot  I(\cdot\, ,\nu):\Bbig(X)\lra \R_{\geq 0}
\]
 are continuous functions.
\end{proposition}
\begin{proof}
The first claim is a consequence of Lebesgue's dominated convergence theorem and the convexity properties of Okounkov functions.  For the second claim, note
that the volume function is continuous and non-zero on $\Bbig(X)$.
\end{proof}

\begin{remark}\label{rem:homogenity of I(D nu)}
Change of variables in the integral and  homogeneity of Okounkov functions yield
\[
 I(mD;\nu) \equ m^{n+1}\cdot I(D;\nu)\ \ \ \text{and}\ \ \ \frac{1}{\vol{X}{mD}}\cdot  I(mD ,\nu) \equ m\cdot \frac{1}{\vol{X}{D}}\cdot  I(D,\nu)\ .
\]
\end{remark}

\section{Appendix: a general 'Fekete lemma' (by S\'ebastien Boucksom)}

\subsection{Facts on semigroups}
Let $V$ be a finite dimensional $\R$-vector space, and $S\subset V$ be a subsemigroup, \ie a non-empty subset stable under taking sums. We denote by:
\begin{itemize}
\item $\Z S=\{s-s'\mid s,s'\in S\}$ the subgroup spanned by $S$,
\item $\R S\subset V$ the $\R$-vector space spanned by $S$,
\item $C(S)\subset\R S$ the convex cone spanned by $S$,
\item $\overline{C}(S)$ its closure, and $\mathring{C}(S)$ its relative interior, \ie its interior in $\R S$.
\end{itemize}
We say that $S$ is a \emph{discrete semigroup} if $\Z S$ is discrete. The \emph{regularization} of $S$ is then defined as the semigroup
$$
S^{\reg}:=\Z S\cap\overline{C}(S).
$$
We rely on the following result, which may be attributed to Khovanskii and appears in \cite{KavKh08} (see also \cite{Bou}).

\begin{proposition}\label{prop:semi} Let $S\subset V$ be a discrete semigroup.
\begin{itemize}
\item[(i)] For every convex cone $\sigma\subset\mathring{C}(S)$ with compact basis, there exists a finitely generated subsemigroup $T\subset S$ such that $S^{\reg}\cap\sigma=T^{\reg}\cap\sigma$.
\item[(ii)] If $T\subset V$ is a discrete semigroup of finite type, then there exists a finite set $F\subset T^{\reg}$ such that $T^{\reg}=T+F$. As a result, $T^{\reg}\setminus T$ meets each cone $\sigma\subset\mathring{C}(T)$ with compact basis in a finite set.
\end{itemize}
\end{proposition}
The first point directly follows from the elementary fact that
$$
\mathring{C}(S)=\bigcup_{T\subset S}\mathring{C}(T),
$$
where $T$ ranges over all finitely generated subsemigroups of $S$. The second point is what the usual proof of Gordan's lemma yields.

\subsection{A 'Fekete lemma' for subadditive functions on semigroups}
If $(a_k)_{k\in\N}$ is a subadditive sequence of real numbers, then $a_k/k$ admits a limit in $\R\cup\{-\infty\}$. This elementary result,
sometimes known as "Fekete's subadditivity lemma", admits the following generalization.

 \begin{thm}\label{thm:subadd} Let $S\subset V$ be a discrete semigroup and $f:S\to\R$ a subadditive function, so that $f(u+v)\le f(u)+f(v)$ for all $u,v\in S$. Then we have:
\begin{itemize}
\item[(i)] For all $x\in\mathring{C}(S)$ and all sequences $\e_ku_k$ with $\e_k>0$, $u_k\in S$, $\e_k\to 0$ and $\e_k u_k\to x$, the limit
$$
\hat f(x)=\lim_{k\to\infty}\e_k f(x_k)
$$
exists in $\R\cup\{-\infty\}$ and only depends on $x$.

\item[(ii)] We either have $\hat f\equiv-\infty$ on $\mathring{C}(S)$, or $\hat f:\mathring{C}(S)\to\R$ is finite valued, homogeneous and subadditive (and hence convex and continuous). In the latter case, we have $\hat f\le f$ on $S\cap\mathring{C}(S)$, and $\hat f$ is characterized as the largest subadditive and homogeneous function on $\mathring{C}(S)$ with this property. 
\end{itemize}
\end{thm}
As observed in \cite[\S 2]{BL}, such a result is implicit in \cite{Zah} for $S=\N^n\subset V=\R^n$. The general case is due to Witt-Nystr\"om \cite{Nys09}, and we will basically follow his strategy of proof.

\begin{proof} Let $\la\in V^*$ be a non-zero linear form, and consider the affine hyperplane $H:=\{\la=1\}$. For all $x\in V$ with $\la(x)\ne 0$, set $\bar x:=\la(x)^{-1}x$, which belongs to $H$. Similarly, for $u\in S$ with $\la(u)\ne 0$ set $\bar f(u):=\la(u)^{-1}f(u)$.

Let
$$
K\Subset K'\Subset K''\Subset\mathring{C}(S)\cap H
$$
be fixed compact convex sets, and denote by $\sigma$, $\sigma'$ and $\sigma''$ the corresponding cones. Note that $\la>0$ on $\sigma''\setminus\{0\}$. To prove (i), it is enough to show that for each $x\in K$ and each sequence $u_k\in S$ with $\la(u_k)\to+\infty$ and $\bar u_k\to x$, $\bar f(u_k)$ has a limit which only depends on $x$.\\

\noindent{\bf Step 1}. We first prove that $\bar f$ is bounded above on $S\cap\sigma$. Applying Proposition \ref{prop:semi} to the discrete semigroup $S\cap\sigma'$, we find finitely many points $u_i\in S\cap\sigma'$ such that $T:=\sum_i\N u_i$ satisfies
$$
S^{\reg}\cap\sigma=T^{\reg}\cap\sigma,
$$
and $T^{\reg}\setminus T$ meets $\sigma$ in a finite set, say $A$. It is thus enough to show that $\bar f$ is bounded above on $(S\cap\sigma)\setminus A$. Now each $u$ in the latter set belongs to $T$, hence writes $u=\sum_i n_i u_i$ with $n_i\in\N$. By subadditivity of $f$, we get
$$
f(u)\le \sum_i n_i f(u_i)\le C\sum_i n_i\la(u_i)=C\la(u)
$$
with $C>0$ larger than $\max_i \la(u_i)^{-1}f(u_i)$, and we thus see that $\bar f\le C$ on $S\cap\sigma$.\\

\noindent{\bf Step 2}. We prove the existence of $C>0$ such that for all $x\in K$ written as the limit of $\bar u_k$ with $u_k\in S$ and $\la(u_k)\to+\infty$, and for all $u'\in S\cap\sigma$, we have
\begin{equation}\label{equ:lip}
\limsup_{k\to\infty}\bar f(u_k)\le\bar f(u')+C\|x-u'\|.
\end{equation}
Given $x$ and $u'$ as above, let $z\in\partial K'$ be the unique point such that $x\in[\bar u',z]$. Since $z$ is in particular in $\mathring{C}(S)$, there exist finitely many points $w_i\in S\cap\sigma''$ such that $\tau:=\sum_i\R_+ w_i$ is a neighborhood of $z$ with
\begin{equation}\label{equ:dist}
d(\tau,K)\ge\tfrac 1 2d(\partial K',K)>0.
\end{equation}
As a result, $\R_+ u'+\sum_i\R_+ w_i$ is a neighborhood of $x$ contained in $\sigma''$,  and we thus have $u_k\in\N u'+\sum_i\N w_i$ for all $k\gg 1$ by (ii) of Proposition \ref{prop:semi}. We may thus write in particular $u_k=m_k u'+r_k$ with $m_k\in\N$ and $r_k\in S\cap\tau$. As a consequence,
$$
t_k:=\frac{\la(r_k)}{\la(u_k)}
$$
belongs to $[0,1]$, and we have
$$
\bar u_k=(1-t_k)\bar u' + t_k\bar r_k
$$
and
$$
\bar f(u_k)\le(1-t_k)\bar f(u')+t_k\bar f(r_k)
$$
by subadditivity of $f$. Applying Step 1 to $K''$ in place of $K$ yields $C>0$ such that $\bar f\le C$ on $S\cap\sigma''$, and we get
$$
\bar f(u_k)-\bar f(u')\le C t_k=C\frac{\|\bar u_k-\bar v\|}{\|\bar r_k-\bar v\|}\le 2C d(\partial K',K)^{-1}\|\bar u_k-\bar v\|
$$
for all $k\gg 1$. This proves (\ref{equ:lip}).\\

\noindent{Step 3}. Let $x\in K$ and let $u_k,u'_k\in S$ be two sequences such that $\la(u_k),\la(u'_k)\to+\infty$ and $\bar u_k,\bar u'_k\to x$.  By (\ref{equ:lip}) we get $\limsup_k\bar f(u_k)\le\liminf_k\bar f(u'_k)$, which proves that $\hat f(x):=\lim_k\bar f(u_k)$ exists in $\R\cup\{-\infty\}$ and only depends on $x$. Another application of (\ref{equ:lip}) shows that $|\hat f(x)-\hat f(x')|\le C\|x-x'\|$ for all $x,x'\in K$, which proves that $\hat f$ is finite valued and continuous on $K$ as soon as there exists $x\in K$ with $\hat f(x)>-\infty$. In that case, subadditivity of $\hat f$ easily follows from that of $f$, and homogeneity of $\hat f$ is automatic, so that $\hat f$ is convex. Given $u\in S\in \mathring{C}(S)$ we have 
$$
\hat f(u)=\lim_{k\to\infty}\tfrac 1 k f(k u)\le f(u)
$$
by subadditivity of $f$. Conversely, if $g$ is a convex and homogeneous function on $\mathring{C}(S)$ such that $g\le f$ on $S\cap\mathring{C}(S)$, writing $x\in\mathring{C}(S)$ as the limit of $\e_k u_k$ with $\e_k\to 0$ and $u_k\in S\in\mathring{C}(S)$ yields
$$
g(u)=\lim_{k\to\infty}\e_k g(u_k)\le\lim_{k\to\infty}\e_k f(u_k)=\hat f(x),
$$
and Theorem \ref{thm:subadd} is proved.

\end{proof}



\bigskip \small

\bigskip\noindent
S{\'e}bastien Boucksom,
CNRS--Universit{\'e} Paris 6, Institut de Math{\'e}matiques, F-75251 Paris Cedex 05, France

\nopagebreak\noindent
\textit{E-mail address:} \texttt{boucksom@math.jussieu.fr}

\bigskip\noindent
   Alex K\"uronya,
   Budapest University of Technology and Economics,
   Mathematical Institute, Department of Algebra,
   Pf. 91, H-1521 Budapest, Hungary.

\nopagebreak\noindent
   \textit{E-mail address:} \texttt{alex.kuronya@math.bme.hu}

\medskip\noindent
   \textit{Current address:}
   Alex K\"uronya,
   Albert-Ludwigs-Universit\"at Freiburg,
   Mathematisches Institut,
   Eckerstra{\ss}e 1,
   D-79104 Freiburg,
   Germany.

\bigskip\noindent
   Catriona Maclean,
   Institut Fourier, CNRS UMR 5582   Universit\'e de Grenoble,
   100 rue des Maths,
   F-38402 Saint-Martin d'H\'eres cedex,  France

\nopagebreak\noindent
   \textit{E-mail address:} \texttt{catriona.maclean@ujf-grenoble.fr}

\bigskip\noindent
   Tomasz Szemberg,
   Instytut Matematyki UP,
   Podchor\c a\.zych 2,
   PL-30-084 Krak\'ow, Poland.

\nopagebreak\noindent
   \textit{E-mail address:} \texttt{tomasz.szemberg@uni-due.de}


\end{document}